\documentclass[11pt]{article}

\usepackage{etex}
\usepackage[all]{xy}
\usepackage{cite}
\usepackage{amsfonts}
\usepackage{amsthm}
\usepackage{enumerate}
\usepackage{graphicx}
\usepackage{mathrsfs}
\usepackage{bm}
\usepackage{amssymb,amsmath} 
\usepackage{makecell}
\usepackage{tikz}
\usepackage{pgf}
\usepackage{tikz}
\usetikzlibrary{patterns}
\usepackage{pgffor}
\usepackage{pgfcalendar}
\usepackage{pgfpages}
\usepackage{shuffle,yfonts}
\usepackage{mathtools}
 \allowdisplaybreaks
\DeclareFontFamily{U}{shuffle}{}
\DeclareFontShape{U}{shuffle}{m}{n}{ <-8>shuffle7 <8->shuffle10}{}

\newcommand{\tn}{{\tilde{n}}}

\newcommand{\dc}{{\rm dc}}
\newcommand{\s}{{\rm sn}}
\newcommand{\nc}{{\rm nc}}
\newcommand{\ds}{{\rm ds}}

\newcommand{\gf}{{\varphi}}

\newcommand\Res{{\rm Res}}

\newcommand\gs{{\sigma}}

 \allowdisplaybreaks
\usetikzlibrary{arrows,shapes,chains}
 \allowdisplaybreaks

\catcode`!=11
\let\!int\int \def\int{\displaystyle\!int}
\let\!lim\lim \def\lim{\displaystyle\!lim}
\let\!sum\sum \def\sum{\displaystyle\!sum}
\let\!sup\sup \def\sup{\displaystyle\!sup}
\let\!inf\inf \def\inf{\displaystyle\!inf}
\let\!cap\cap \def\cap{\displaystyle\!cap}
\let\!max\max \def\max{\displaystyle\!max}
\let\!min\min \def\min{\displaystyle\!min}
\let\!frac\frac \def\frac{\displaystyle\!frac}
\catcode`!=12

\let\oldsection\section
\renewcommand\section{\setcounter{equation}{0}\oldsection}

\allowdisplaybreaks

\def\N{\mathbb{N}}
\def\Z{\mathbb{Z}}
\def\Q{\mathbb{Q}}

\def\cn{{\rm cn}}
\def\dn{{\rm dn}}
\def\su{\sum\limits_{n=1}^\infty}

\theoremstyle{plain}
\newtheorem{thm}{Theorem}[section]
\newtheorem{lem}[thm]{Lemma}
\newtheorem{cor}[thm]{Corollary}
\newtheorem{con}[thm]{Conjecture}
\newtheorem{pro}[thm]{Proposition}
\theoremstyle{definition}

\newtheorem{re}[thm]{Remark}
\newtheorem{exa}[thm]{Example}

\begin{document}
\title{\bf Berndt-Type Integrals and Series Associated with Ramanujan and Jacobi Elliptic Functions}
\author{
{Ce Xu${}^{a,}$\thanks{Email: cexu2020@ahnu.edu.cn}\quad and \quad Jianqiang Zhao${}^{b,}$\thanks{Email: zhaoj@ihes.fr}}\\[1mm]
\small a. School of Mathematics and Statistics, Anhui Normal University,  Wuhu 241002, P.R. China\\
\small b. Department of Mathematics, The Bishop's School, La Jolla, CA 92037, USA\\[5mm]
\normalsize \emph{Dedicated to Professor Bruce C. Berndt on the occasion of his 85th birthday}
}

\date{}
\maketitle

\noindent{\bf Abstract.} We evaluate two families of infinite integrals containing hyperbolic and trigonometric functions in closed forms in this paper. We call them Berndt-type integrals since Berndt initiated the study of similar integrals previously. We first compute four classes of hyperbolic sums by special values of the Gamma function, by two completely different approaches, which extend those sums considered by Ramanujan and Zucker previously. We carry out the first by refining two results of Ramanujan concerning some $q$-series. For the second we compare both the Fourier series expansions and the Maclaurin series expansions of a few Jacobi elliptic functions. Next, by contour integrations we convert two families of Berndt-type integrals to the above hyperbolic sums, all of which can be evaluated in closed forms. We then discover explicit formulas for one of the two families. Throughout the paper we present many examples which enable us to formulate a conjectural explicit formula for the other family of the Berndt-type integrals at the end.




\section{Introduction}
This paper is originally motivated by  \cite{Berndt2016} in which Berndt proved some explicit relations between hyperbolic series and integrals containing trigonometric and hyperbolic functions. For example, from \cite[Thm. 3.1 and Eq. (5.5)]{Berndt2016} for nonnegative integer $a$, Berndt showed that
\begin{align}\label{equ.one}
(1+(-1)^a)\int_0^\infty \frac{x^{2a+1}dx}{\cos x+\cosh x}=\frac{\pi^{2a+2}i^a}{2^a} \sum_{n=0}^\infty \frac{(-1)^n (2n+1)^{2a+1}}{\cosh((2n+1)\pi/2)},
\end{align}
and
\begin{align}\label{equ.-.one}
(1+i^{a+1})\int_0^\infty \frac{x^{a}dx}{\cos x-\cosh x}=2i(1+i)^{a-1}\pi^{a+1}\sum_{n=1}^\infty \frac{(-1)^{n+1}n^a}{\sinh(n\pi)}\quad (a\geq 2).
\end{align}
In his second notebook \cite{R2012} (see \cite[p.134, Entry 16(i)-(vi)]{B1991}), for odd integers $s<12$ Ramanujan evaluated the
hyperbolic sums
\begin{align}\label{equ.Ramanujan-iden-s}
\sum_{n=0}^\infty \frac{(-1)^n (2n+1)^{s}}{\cosh((2n+1)y/2)}
\end{align}
in terms of $x$ and $z$. For example, it is proved that
\begin{align}\label{equ.Ramanujan-iden}
\sum_{n=0}^\infty \frac{(-1)^n (2n+1)^3}{\cosh((2n+1)y/2)}=\frac1{4}z^4(1-2x)\sqrt{x(1-x)},
\end{align}
where the definition of $x,y$ and $z$ are given in \eqref{rel-den}. Zucker \cite{Z1979} extended these evaluations by replacing $\cosh$ with $\sinh$ in \eqref{equ.Ramanujan-iden-s}.

In this paper, we apply some carefully chosen contour integrations related to hyperbolic and trigonometric functions to build on the previous work on Berndt-type integrals and hyperbolic infinite sums (see \cite{Berndt2016,Campbell,Rama1916,Ya-2018,Z1979}).
First, we will prove the following general results.

\begin{thm} \label{thm-main1}
Let $\Gamma(x)$ be the usual gamma function. For any positive integer $m$,
\begin{align*}
\sum_{n=1}^\infty \frac{n^{4m-2}}{\sinh^2(n\pi)}  &\,\in \frac{\Gamma^{8m}(1/4)}{\pi^{6m}}\Q+\frac{\delta_{m,1}}{\pi^2}\Q, \\ \sum_{n=1}^\infty \frac{n^{4m-4}}{\sinh^2(n\pi)}  &\, \in  \frac{\Gamma^{8m-8}(1/4)}{\pi^{6m-5}} \Q+\frac{\Gamma^{8m-4}(1/4)}{\pi^{6m-3}}\Q+\delta_{m,1}\Q,  \\
\sum_{n=1}^\infty \frac{n^{4m-1}\cosh(n\pi)}{\sinh^3(n\pi)} &\,\in
\frac{\Gamma^{8m}(1/4)}{\pi^{6m+1}}\Q+\frac{\delta_{m,1}}{\pi^3}\Q, \\
\sum_{n=1}^\infty \frac{n^{4m-3}\cosh(n\pi)}{\sinh^3(n\pi)}  &\,\in \frac{\Gamma^{8m}(1/4)}{\pi^{6m}}\Q
+\frac{\Gamma^{8m-8}(1/4)}{\pi^{6m-4}}\Q,
\end{align*}
where $\delta_{a,b}$ is the Kronecker symbol.
\end{thm}


\begin{thm} \label{thm-main2}
For any integer $m\ge 0$,
\begin{align*}
\sum_{n=0}^\infty \frac{(2n+1)^{4m+2}}{\cosh^2((2n+1)\pi/2)}  &\, \in \frac{\Gamma^{8m+8}}{\pi^{6m+6}}\Q ,\\
\sum_{n=0}^\infty \frac{(2n+1)^{4m}}{\cosh^2((2n+1)\pi/2)}    &\, \in \frac{\Gamma^{8m}}{\pi^{6m+1}}\Q, \\
\sum_{n=0}^\infty \frac{(2n+1)^{4m+1}\sinh((2n+1)\pi/2)}{\cosh^3((2n+1)\pi/2))}
    &\,  \in \frac{\Gamma^{8m+8}}{\pi^{6m+6}}\Q+\frac{\Gamma^{8m}}{\pi^{6m+2}}\Q, \\
\sum_{n=0}^\infty \frac{(2n+1)^{4m+3}\sinh((2n+1)\pi/2)}{\cosh^3((2n+1)\pi/2))} &\,  \in  \frac{\Gamma^{8m+8}}{\pi^{6m+7}}\Q.
\end{align*}
\end{thm}
A more precise version of the above theorem is given by Theorem \ref{thm-main2Precise} in which
all the rational constants are given explicitly as a function of $m$ in closed forms.

In another direction, the Ramanujan's identity \eqref{equ.Ramanujan-iden} was applied by Berndt in \cite{Berndt2016} to prove the following beautiful identity:
\begin{align}\label{equ.Ramanujan-iden-inte}
\int_{-\infty}^\infty \frac{dx}{\cos\sqrt{x}+\cosh\sqrt{x}}=\frac{\pi}{4}\frac{\Gamma^2(1/4)}{\Gamma^2(3/4)}=\frac{\Gamma^4(1/4)}{2\pi}.
\end{align}
In fact, at the \emph{International Conference on Orthogonal Polynomials and $q$-Series}, which was held in Orlando in May 2015 in celebration of the 70th birthday of Professor Mourad Ismail, Dennis Stanton gave a plenary talk titled ``A small slice of Mourad's work". One of the topics in that talk was about ``the mystery integral of Mourad Ismail" \cite{K2017}:
\begin{align}\label{inte-Mourad-Ismail}
\int_{-\infty}^\infty \frac{dt}{\cos(\sqrt{t} K(x) )+\cosh( \sqrt{t}K'(x))}=2,
\end{align}
where $K(x)$ denotes the complete elliptic integral of the first kind, see Section \ref{sec-prelim}. This curious integral \eqref{inte-Mourad-Ismail} first appeared in \cite{Ismail1998} by Ismail and Valent. In \cite{K2017}, Kuznetsov proved a more general result
\begin{align*}
\frac1{2}\int_{-\infty}^\infty \frac{t^n dt}{\cos(\sqrt{t}K(x))+\cosh(\sqrt{t}K'(x))}=(-1)^n \frac{d^{2n+1}}{du^{2n+1}} \frac{{\rm sn} (u,k)}{{\rm cd}(u,k)}|_{u=0},
\end{align*}
where the Jacobi elliptic function ${\rm sn} (u)$ is defined via the inversion of the elliptic integral, see Section \ref{sec-Jacobi Functions}.

Some recent results on infinite series involving hyperbolic functions may be found in the works of \cite{BB2002,T2015,T2008,T2010,T2012,X2018,XuZhao-2022,Ya-2018} and the references therein.
Using Theorem \ref{thm-main1}, we will prove the following evaluations (see Theorems \ref{thm-main3Precise} and \ref{thm-main4Precise}).

\begin{thm}  \label{thm-main34}
For any positive integer $p$,
\begin{align}
&\int_0^\infty \frac{x^{4p+1}}{(\cos x-\cosh x)^2}dx=i_p\frac{\Gamma^{8p}(1/4)}{\pi^{2p}}+j_p\frac{\Gamma^{8p+8}(1/4)}{\pi^{2p+4}},\label{conj-for-a2}\\
&\int_0^\infty \frac{x^{4p+1}}{(\cos x+\cosh x)^2}dx=g_p\frac{\Gamma^{8p}(1/4)}{\pi^{2p}}+h_p\frac{\Gamma^{8p+8}(1/4)}{\pi^{2p+4}},\label{conj-for-a1}
\end{align}
where $g_p,h_p\in \Q$ are given explicitly in closed forms in Theorem \ref{thm-main4Precise},
and $i_p,j_p\in \Q$ can be explicitly computed for each fixed $p$.
\end{thm}

Two natural questions are how to evaluate the following more general Berndt-type integrals and Ismail-type integrals?

\medskip\noindent
{\bf Question one:} For integers $a\geq 0$ and $b\geq 1$, can we express the following Berndt-type integrals in closed forms:
\begin{align}\label{question-1}
\int_0^\infty \frac{x^{a}dx}{(\cos x+\cosh x)^b}=?
\end{align}
And if $a\geq 2b$
\begin{align}\label{question-2}
\int_0^\infty \frac{x^{a}dx}{(\cos x-\cosh x)^b}=?
\end{align}
{\bf Question two:} For integers $a\geq 0$ and $b\geq 1$, can we express the following Ismail-type integrals in closed forms:
\begin{align}\label{question-3}
\int_{-\infty}^\infty \frac{t^adt}{(\cos(\sqrt{t}K(x) )+\cosh(\sqrt{t}K'(x)))^b}=?
\end{align}
And if $a\geq 2b$
\begin{align}\label{question-4}
\int_{-\infty}^\infty \frac{t^adt}{(\cos(\sqrt{t}K(x))-\cosh(\sqrt{t}K'(x)))^b}=?
\end{align}

\section{Brief review of complete  elliptic integral of the first kind}\label{sec-prelim}
To fix notation we now give a very brief review of the complete  elliptic integral of the first kind in this section. As usual, $(a)_n$ denotes the ascending Pochhammer symbol defined by
\begin{align}
(a)_n:=\frac {\Gamma(a+n)}{\Gamma(a)}=a(a+1)\cdots(a+n-1) \quad {\rm and}\quad (a)_0:=1,
\end{align}
where $a\in \mathbb{C}$ is any complex number and $n\in\N_0$ is a nonnegative integer. Let $\Gamma \left(z\right)$ be
the Gamma function as usual. If $\Re(z)>0$ then
\begin{equation*}
\Gamma(z) := \int_0^\infty e^{-t}t^{z-1} \, dt.
\end{equation*}

The Gaussian or ordinary hypergeometric function ${_2}F_1(a,b;c;x)$ is defined for $|x|<1$ by the power series
\begin{align}
{_2}F_1(a,b;c;x)=\sum\limits_{n=0}^\infty \frac{(a)_n(b)_n}{(c)_n} \frac {x^n}{n!}\quad (a,b,c\in\mathbb{C}).
\end{align}
The complete elliptic integral of the first kind is defined by (see \cite{WW1966})
\begin{align}\label{defnK}
K:=K(x):=K(k^2):=\int\limits_{0}^{\pi/2}\frac {d\varphi}{\sqrt{1-k^2\sin^2\varphi}}=\frac {\pi}{2} {_2}F_{1}\left(\frac {1}{2},\frac {1}{2};1;k^2\right).
\end{align}
Here $x=k^2$ and $k\ (0<k<1)$ is the modulus of $K$. The complementary modulus $k'$ is defined by $k'=\sqrt {1-k^2}$.
Furthermore, we set as usual
\begin{align*}
K':=K(k'^2)=\int\limits_{0}^{\pi/2}\frac {d\varphi}{\sqrt{1-k'^2\sin^2\varphi}}=\frac {\pi}{2} {_2}F_{1}\left(\frac {1}{2},\frac {1}{2};1;1-k^2\right).
\end{align*}
In order to better state our main results, we shall henceforth adopt the notations of Ramanujan (see Berndt's book \cite{B1991}). Let
\begin{align}\label{rel-den}
x:=k^2,\quad y:=y(x):=\pi \frac {K'}{K}, \quad q:=q(x):=e^{-y},\quad z:=z(x):=\frac {2}{\pi}K,\quad z':=\frac{dz}{dx}.
\end{align}

The following results will be very useful for our future computation.
\begin{lem}
For all nonnegative integers $n\in\N_0$, we have
\begin{align}\label{den-z-diff}
\frac {d^nz}{dx^n}\bigg|_{x=1/2}=\frac{(1/2)^2_n\sqrt{\pi}}{\left(\Gamma\left(\frac{n}{2}+\frac {3}{4}\right)\right)^2}.
\end{align}
\end{lem}
\begin{proof}
This is \cite[(2.2)]{RXYZ2024}.
\end{proof}

\section{Refined versions of two results of Ramanujan's}
In \cite{Rama1916} Ramanujan studied the $q$-series
\begin{equation*}
\Phi_{a,m}(q):=\sum_{j,n=1}^\infty n^m j^a q^{jn}
\end{equation*}
for integers $a\ge 0$ and $m\ge 1$. If $a=0,1$ then
\begin{equation*}
\Phi_{a,m}(q)=\sum_{n=1}^\infty \frac{n^m q^{n}}{(1-q^{n})^{a+1}}.
\end{equation*}
We follow Ramanujan's notation by setting
\begin{equation*}
S_m:=\frac12 \zeta(-m)+\Phi_{0,m} =-\frac{B_{m+1}}{2(m+1)}  +\Phi_{0,m}  \quad\forall r\ge 1.
\end{equation*}
The three relevant Eisenstein series are defined for $|q|<1$ by
\begin{align*}
&P:=1-24\su \frac{nq^{2n}}{1-q^{2n}},\quad
Q:=1+240\su \frac{n^3q^{2n}}{1-q^{2n}},\quad
R:=1-504\su \frac{n^5q^{2n}}{1-q^{2n}}.
\end{align*}
By Entry 13 in Chapter 17 of Ramanujan's third notebook \cite{B1991}, we have
\begin{align}
&P=(1-2x)z^2+6x(1-x)zz',\label{3.583}\\
&Q=z^4(1-x+x^2),\label{3.55}\\
&R=z^6(1+x)(1-x/2)(1-2x).\label{3.57}
\end{align}
Ramanujan proved that the $q$-series $\Phi_{a,m}\in\Q[P,Q,R].$
In order to prove Theorem \ref{thm-main1} we need to refine this result when $a=0$ and $a=1$.

\begin{pro} \label{pro:RamanujanS}
Let $\gs=x(1-x)$. Then $S_1\in z^2\Q[\gs]\gs'+z z'\Q[\gs]$. For all $m\ge 1$ we have
\begin{alignat}{3}
 &S_{4m-1}\in z^{4m}\Q[\gs],  \quad  \quad  & & S'_{4m-1} \in z^{4m}\Q[\gs]\gs'+z^{4m-1}z'\Q[\gs] ,  \label{equ:S}\\
 &S_{4m+1} \in z^{4m+2}\Q[\gs]\gs', \quad\quad  & & S'_{4m+1} \in z^{4m+2}\Q[\gs]+ z^{4m+1}z'\Q[\gs]\gs'. \label{equ:S'}
\end{alignat}
\end{pro}
\begin{proof}
First note that $S_1= -P/24=-z^2\gs'/24-zz'\gs/4 \in z^2\Q[\gs]\gs'+z z'\Q[\gs]$.
Furthermore,
\begin{alignat*}{4}
S_3=&\, Q/240=z^4(1-\gs)/240, \quad & S_3'=&\, z^3z'(1-\gs)/60-z^4\gs'/240, \\
S_5=&\, -R/504=z^6(2+\gs)\gs'/1008,\quad & S_5'=&\, z^5z'(2+\gs)\gs'/168-z^6(1+2\gs)/336.
\end{alignat*}
Hence the claim is true for $m=1$. By \cite[(22)]{Rama1916},
\begin{equation} \label{equ:Srecur}
\frac{(2n+5)(n-1)}{12(2n+1)(n+1)} S_{2n+3}=\sum_{k=1}^{n-1}  \binom{2n}{2k} S_{2k+1} S_{2n-2k+1}.
\end{equation}
Suppose $n=2m-1$ is odd, then the sum of the two indices on the right-hand side of the above is $4m$.
Thus one index must have the form $4m_1+1$ and the other $4m_2-1$ with $m_1+m_2=m$.
Therefore $S_{4m+1} \in z^{4m_1+4m_2+2}\Q[\gs]\gs'=z^{4m+2}\Q[\gs]\gs'$ by induction assumption.

If $n=2m-2\ge 2$ is even, then the sum of the two indices on the right-hand side of the above is $4m-2$.
Thus either (i) $2k+1=4m_1-1>1$ and $2n-2k+1=4m_2-1>1$ with $m_1+m_2=m$, or
(ii) $2k+1=4m_1+1>1$ and $2n-2k+1=4m_2+1>1$ with $m_1+m_2=m-1$. If (i) holds then $S_{4m-1}\in z^{4m_1+4m_2}\Q[\gs]=z^{4m}\Q[\gs]$.
If (ii) holds then $S_{4m-1}\in z^{4m_1+2+4m_2+2}\Q[\gs](\gs')^2=z^{4m}\Q[\gs]$ since $(\gs')^2=1-4\gs$.

This completes the proof of the claims for $S_{\text{odd}}$ by induction.
The claims on the derivatives follow immediately.
\end{proof}

\begin{pro} \label{pro:RamanujanPhi1}
Let $\gs=x(1-x)$. Then we have
\begin{align}\label{equ:Phi12}
\Phi_{1,2} \in &\, z^4\Q[\gs]+z^3z'\Q[\gs]\gs'+z^2(z')^2\Q[\gs] ,\\
\Phi'_{1,2}\in &\, z^4\Q[\gs]\gs'+z^3z'\Q[\gs]+z^3z''\Q[\gs]\gs'+z^2(z')^2\Q[\gs]\gs'+z^2z'z''\Q[\gs]+z(z')^3\Q[\gs],\label{equ:Phi'12}
\end{align}
and for all $m\ge 1$
\begin{align}
 \Phi_{1,4m+2} &\, \in z^{4m+4}\Q[\gs]+z^{4m+3}z'\Q[\gs]\gs',   \quad
 \Phi_{1,4m}\in z^{4m+2}\Q[\gs]\gs'+z^{4m+1}z'\Q[\gs],  \label{equ:Phi}\\
 \Phi'_{1,4m+2}&\, \in  z^{4m+4}\Q[\gs]\gs'+z^{4m+3}z'\Q[\gs]+z^{4m+3}z''\Q[\gs]\gs'+z^{4m+2}(z')^2\Q[\gs]\gs', \label{equ:Phi'4m+2} \\
 \Phi'_{1,4m} &\, \in z^{4m+2}\Q[\gs]+z^{4m+1}z'\Q[\gs]\gs'+z^{4m+1}z''\Q[\gs]+z^{4m}(z')^2\Q[\gs]. \label{equ:Phi'4m}
\end{align}
\end{pro}
\begin{proof} First we note that from \cite[Table II]{Rama1916}
\begin{equation*}
288\Phi_{1,2}= Q-P^2 = 3z^4 \gs-12 z^3z'\gs\gs'-36 z^2(z')^2 \gs^2 \in z^4\Q[\gs]+z^3z'\Q[\gs]\gs'+z^2(z')^2\Q[\gs].
\end{equation*}
which yields \eqref{equ:Phi12}. Differentiating with respect to $x$ and using the fact that
$\gs''=-2$ and $(\gs')^2=1-4\gs$ we obtain \eqref{equ:Phi'12} immediately.

When $m=1$, by \cite[Table II]{Rama1916}, we see that
\begin{align*}
1440\Phi_{1,4}=&\,2(PQ-R)=-3z^6 \gs\gs'+12z^5z'(1-\gs)\gs,\\
2016 \Phi_{1,6}=&\, 2(Q^2-PR)=z^8 (4\gs^2+3\gs)-z^7z'(2-\gs)\gs\gs'.
\end{align*}
Thus the proposition holds for $m=1$.

By \cite[(28)]{Rama1916}, for all $n\ge 1$
\begin{equation}\label{equ:PhiInd}
\frac{2n+3}{2(2n+1)}S_{2n+1}-\Phi_{1,2n}=\sum_{k=1}^{n}  \binom{2n}{2k-1} S_{2k-1} S_{2n-2k+1}.
\end{equation}
Suppose $n=2m+1\ge 3$. Then by Prop.~\ref{pro:RamanujanS} we see that $S_{2n+1}=S_{4m+3}\in z^{4m+4}\Q[\gs]$
while
\begin{equation*}
S_1S_{2n-1}=S_1S_{4m+1}\in (z^2\Q[\gs]\gs'+z z'\Q[\gs])z^{4m+2}\Q[\gs]\gs'\subset z^{4m+4}\Q[\gs]\gs'+z^{4m+3}z'\Q[\gs]
\end{equation*}
since $(\gs')^2=1-4\gs$. The other products on the right-hand side of \eqref{equ:PhiInd}
are either of the form $z^{4m_1+4m_2+4}\Q[\gs](\gs')^2$ (when $k$ is odd, $m_1+m_2=m$)
or of the form $^{4m_1+4m_2+8}\Q[\gs]$ (when $k$ is even, $m_1+m_2=m-1$). Thus
$\Phi_{1,4m+2} \in z^{4m+4}\Q[\gs]\gs'+z^{4m+3}z'\Q[\gs]$.

If $n=2m$, then $S_{4m+1} \in z^{4m+2}\Q[\gs]\gs'$ by Prop.~\ref{pro:RamanujanS} and
\begin{equation*}
S_1S_{2n-1}=S_1S_{4m-1}\in (z^2\Q[\gs]\gs'+z z'\Q[\gs])z^{4m}\Q[\gs]\subset z^{4m+2}\Q[\gs]\gs'+z^{4m+1}z'\Q[\gs].
\end{equation*}
The other products must have the form of $z^{4m_1+4m_1+2}\Q[\gs]\gs'$ for some $m_1+m_2=m$. Thus
$\Phi_{1,4m}\in   z^{4m+2}\Q[\gs]\gs'+z^{4m+1}z'\Q[\gs]$.
This completes the proof of \eqref{equ:Phi} of the proposition by induction.

Finally, \eqref{equ:Phi'4m+2} and  \eqref{equ:Phi'4m} follow easily from \eqref{equ:Phi}.
\end{proof}

\section{Relations between Berndt-type Integrals and Series}
In this section, we will establish some explicit relations between the Berndt-type integrals \eqref{question-1}, \eqref{question-2} with $b=2$ and some hyperbolic summations by using the contour integrations. To save space, throughout this section we will put
\begin{equation*}
\tn=\frac{2n-1}2.
\end{equation*}

\begin{lem}\label{lem-tri-hy}
 \emph{(cf. \cite{X2018})} Let $n$ be an integer. Then we have
\begin{align}
&\frac{\pi}{\sinh(\pi s)} \buildrel{s \to ni}\over{=} (-1)^n \left(\frac{1}{s-ni}
+ 2\sum_{k = 1}^\infty  (-1)^k\bar\zeta(2k)(s - ni)^{2k-1} \right),\label{e5}
\\
& \frac{\pi}{\cosh(\pi s)}\buildrel{s \to \tn i}\over{=} (-1)^n i\left(\frac{1}{s-\tn i}
+ 2\sum_{k = 1}^\infty  (-1)^k\bar\zeta(2k)(s - \tn i)^{2k-1} \right),\label{e9}
\end{align}
where ${\bar \zeta}(s)$ denotes the alternating Riemann zeta function defined by
\begin{align*}
\bar \zeta \left( s \right) := \sum\limits_{n = 1}^\infty  {\frac{{{{\left( { - 1} \right)}^{n - 1}}}}{{{n^s}}}}\quad ( \Re(s) >0).
\end{align*}
\end{lem}

Recall that the Riemann zeta function $\zeta(s)$  is defined by
\begin{align*}
\zeta(s):=\sum\limits_{n = 1}^\infty {\frac {1}{n^{s}}}\quad (\Re(s)>1).
\end{align*}
When $s=2m\ (m\in \N)$ is an even number Euler proved the famous formula
\begin{align}
\zeta(2m) = -\frac12\frac{B_{2m}}{(2m)!}(2\pi i)^{2m},
\end{align}
where $B_{2m}$ is Bernoulli number which are defined by the generating function
\begin{equation*}
 \frac{x}{e^x-1}=\sum_{n=0}^\infty \frac{B_n}{n!}x^n.
\end{equation*}
In particular, $B_0=1,B_1=-\frac1{2},B_2=\frac1{6},B_4=-\frac1{30},B_6=\frac1{42}$ and $B_8=-\frac{1}{30}$.

From Lemma \ref{lem-tri-hy}, by direct calculations, we obtain
\begin{align}
&\frac{(-1)^n}
{{\cosh \left(\frac{1+i}{2}z \right)}}
\mathop  = \limits^{z \to \tn\pi (1+i)}  2i \left\{\frac1{1+i}\cdot\frac{1}{z-\tn\pi(1+i)}\atop +\sum_{k=1}^\infty (-1)^k \frac{{\bar \zeta}(2k)}{\pi^{2k}}\left(\frac{1+i}{2}\right)^{2k-1}\left(z-\tn\pi(1+i)\right)^{2k-1} \right\},\label{eq-cosh-1}\\
&\frac{(-1)^n}
{{\cosh \left(\frac{1-i}{2}z \right)}}
\mathop  = \limits^{z \to \tn\pi (i-1)}   2i \left\{\frac1{1-i}\cdot\frac{1}{z-\tn\pi(i-1)}\atop +\sum_{k=1}^\infty (-1)^k \frac{{\bar \zeta}(2k)}{\pi^{2k}}\left(\frac{1-i}{2}\right)^{2k-1}\left(z-\tn\pi(i-1)\right)^{2k-1} \right\},\label{eq-cosh-2}\\
&\frac{(-1)^n}
{{\sinh \left(\frac{1+i}{2}z \right)}}
\mathop  = \limits^{z \to n\pi (1+i)}  2\left\{\frac1{1+i}\cdot\frac{1}{z-n\pi(1+i)}\atop +\sum_{k=1}^\infty (-1)^k \frac{{\bar \zeta}(2k)}{\pi^{2k}}\left(\frac{1+i}{2}\right)^{2k-1}\left(z-n\pi(1+i)\right)^{2k-1} \right\},\label{eq-sinh-1}\\
&\frac{(-1)^n}
{{\sinh \left(\frac{i-1}{2}z \right)}}
\mathop  = \limits^{z \to n\pi (1-i)}  2\left\{\frac1{i-1}\cdot\frac{1}{z-n\pi(1-i)}\atop +\sum_{k=1}^\infty (-1)^k \frac{{\bar \zeta}(2k)}{\pi^{2k}}\left(\frac{i-1}{2}\right)^{2k-1}\left(z-n\pi(1-i)\right)^{2k-1} \right\}.\label{eq-sinh-2}
\end{align}
The notation \(\stackrel{z \to m\pi(a+bi)}{=}\) appearing in the above formula, with \(m \in \{n,\tilde{n}\}\), means that \(z\) ranges over a sufficiently small punctured neighborhood centered at \(m\pi(a+bi)\).

\begin{thm}\label{thm-Integ-series-1} For any $a\geq 0$, we have
\begin{align*}
\frac{(1-i^{a+1})}{\pi^a (1+i)^{a-1}}\int_0^\infty \frac{x^{a}dx}{(\cos x+\cosh x)^2}
&=2\pi\sum_{n=1}^\infty \frac{\tn^a\sinh(\tn \pi)}{\cosh^3(\tn \pi)}
 -a\sum_{n=1}^\infty \frac{\tn^{a-1}}{\cosh^2(\tn \pi)}.
\end{align*}
\end{thm}
\begin{proof}
Consider
\begin{align*}
\lim_{R\rightarrow\infty }\int_{C_R} \frac{z^{a}dz}{(\cos z+\cosh z)^2}=\lim_{R\rightarrow\infty }\int_{C_R} F(z)dz,
\end{align*}
where $C_R$ denotes the positively oriented quarter-circular contour consisting of the interval $[0,R]$, the quarter-circle $\Gamma_R$ with $|z|=R$ and $0\leq \arg\leq \pi/2$, and $[iR,0]$ (i.e., the line segment from $i R$ to $0$ on the imaginary axis), please see the following figure.
\begin{center}
\includegraphics[height=2in]{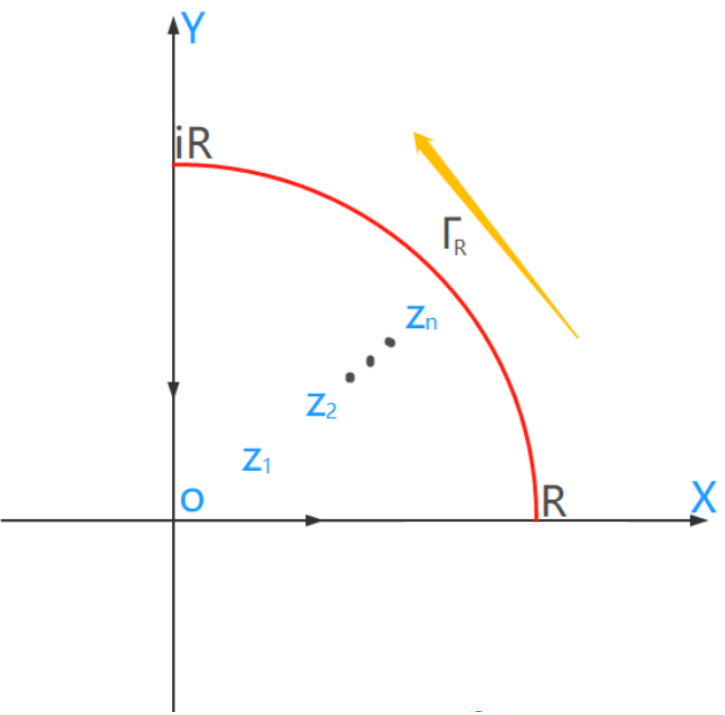}
\end{center}
Clearly, there exist poles of order $2$ in the present case when
\begin{align*}
\cos z+\cosh z=2\cos\left\{\frac1{2}(z+iz)\right\}\cos\left\{\frac1{2}(z-iz)\right\}=0.
\end{align*}
Hence, there exist one set of poles at
\begin{align*}
z_n:=\frac{(2n-1)\pi i}{1+i}=\tn (1+i)\pi ,\quad n\in \Z,
\end{align*}
and
\begin{align*}
s_n:=\frac{(2n-1)\pi i}{1-i}=\tn(i-1)\pi ,\quad n\in \Z.
\end{align*}
The only poles lying inside $C_R$ are $z_n,\ n\geq 1,\ |z_n|<R$. From \eqref{eq-cosh-1}, we have
\begin{align*}
&\frac{1}
{{\cosh \left(\frac{1+i}{2}z \right)}}\mathop  = \limits^{z \to z_n} (-1)^n i \left\{\frac2{1+i}\cdot\frac{1}{z-z_n}-\frac{1+i}{12}\left(z-z_n\right)+o(z-z_n)\right\}.
\end{align*}
Further, one obtains
\begin{align*}
&\frac{1}
{{\cosh^2 \left(\frac{1+i}{2}z \right)}}\mathop  = \limits^{z \to z_n} -\frac4{(1+i)^2} \frac{1}{(z-z_n)^2}+\frac1{3}+o(1).
\end{align*}
Hence, if $z\rightarrow z_n$ then
\begin{align*}
&F(z)\mathop  = \limits^{z \to z_n} \frac{z^a}{4{\cosh^2 \left(\frac{1-i}{2}z \right)}}\left\{-\frac4{(1+i)^2} \frac{1}{(z-z_n)^2}+\frac1{3}+o(1)\right\}.
\end{align*}
The residue $\Res[F(z),z_n]$ at each such pole is give by
\begin{align}\label{residue-case-1}
\Res[F(z),z_n]&=\lim_{z\rightarrow z_n}\frac{d}{dz} \left\{(z-z_n)^2 \frac{z^a}{(\cos z+\cosh z)^2}\right\}\nonumber\\
&=\lim_{z\rightarrow z_n}\frac{d}{dz} \frac{z^a}{4{\cosh^2 \left(\frac{1-i}{2}z \right)}}\left\{-\frac4{(1+i)^2}+\frac1{3}(z-z_n)^2+o((z-z_n)^2)\right\}\nonumber\\
&=-a(1+i)^{a-3}\frac{\left(\tn\pi\right)^{a-1}}{\cosh^2\left(\tn\pi\right)}+(1-i)(1+i)^{a-2} \left(\tn\pi\right)^a \frac{\sinh\left(\tn\pi\right)}{\cosh^3\left(\tn\pi\right)}.
\end{align}
For the integral over $[iR,0]$, we let $z=ir$. Hence,
\begin{align}\label{residue-case-2}
\int_{iR}^0 \frac{z^adz}{(\cos z+\cosh z)^2}=\int_R^0 \frac{i^{a+1}rdr}{(\cos ir+\cosh ir)^2}=-i^{a+1} \int_0^R \frac{r^adr}{(\cos r+\cosh r)^2}.
\end{align}
It is easy to show that, as $R\rightarrow \infty$
\begin{align}\label{residue-case-3}
\int_{\Gamma_R}\frac{z^adz}{(\cos z+\cosh z)^2}=o(1).
\end{align}
Applying the residue theorem, letting $R\rightarrow \infty$, and collecting \eqref{residue-case-1}-\eqref{residue-case-3}, we conclude that
\begin{align}\label{residue-case-4}
(1-i^{a+1})\int_{0}^\infty \frac{z^adz}{(\cos z+\cosh z)^2}=2\pi i \sum_{n=1}^\infty \Res[F(z),z_n].
\end{align}
This completes the proof of the theorem.
\end{proof}

Letting $a=1,3$ yields
\begin{align}
&2\int_0^\infty \frac{xdx}{(\cos x+\cosh x)^2}=-\pi  \sum_{n=1}^\infty \frac{1}{\cosh^2\left(\tn\pi\right)}+\pi^2\sum_{n=1}^\infty \frac{(2n-1)\sinh\left(\tn\pi\right)}{\cosh^3\left(\tn\pi\right)}
\end{align}
and
\begin{align}\label{case-2}
&0=-3\sum_{n=1}^\infty \frac{\left(\tn\pi\right)^2}{\cosh^2\left(\tn\pi\right)}+2\sum_{n=1}^\infty \frac{\left(\tn\pi\right)^3\sinh\left(\tn\pi\right)}{\cosh^3\left(\tn\pi\right)}.
\end{align}
From \cite{XuZhao-2022}, we find that
\begin{align*}
\sum_{n=1}^\infty \frac{(2n-1)^2}{\cosh^2\left(\tn\pi\right)}=\frac{\Gamma^8(1/4)}{192\pi^6}.
\end{align*}
Hence, together with \eqref{case-2} this gives
\begin{align}\label{new-case-1}
\sum_{n=1}^\infty \frac{(2n-1)^3\sinh\left(\tn\pi\right)}{\cosh^3\left(\tn\pi\right)}=\frac{\Gamma^8(1/4)}{64\pi^7}.
\end{align}

\begin{thm}\label{thm-Integ-series-2} For any $a\geq 4$, we have
\begin{align*}
\frac{(1-i^{a+1})}{\pi^a(1+i)^{a-1}}\int_0^\infty \frac{x^{a}dx}{(\cos x-\cosh x)^2}
=a \sum_{n=1}^\infty \frac{n^{a-1}}{\sinh^2(n\pi)}
 -2\pi \sum_{n=1}^\infty \frac{n^a\cosh(n \pi)}{\sinh^3(n\pi)}.
\end{align*}
\end{thm}
\begin{proof}
Consider
\begin{align*}
\lim_{R\rightarrow\infty }\int_{C_R} \frac{z^{a}dz}{(\cos z-\cosh z)^2}=\lim_{R\rightarrow\infty }\int_{C_R} G(z)dz,
\end{align*}
where $C_R$ denotes the positively oriented quarter-circular contour consisting of the interval $[0,R]$; the quarter-circle $\Gamma_R$ with $|z|=R$ and $0\leq \arg\leq \pi/2$; and the segment $[i R,0]$ on the imaginary axis. Clearly, there exist poles of order $2$ in the present case when
\begin{align*}
\cos z-\cosh z=-2\sin\left\{\frac1{2}(z+iz)\right\}\sin\left\{\frac1{2}(z-iz)\right\}=0.
\end{align*}
Hence, there exist one set of poles at $t_n=n\pi(1-i)$ and $y_n=n\pi(1+i)$, where $n\in\N$. Those lying on the interior of $C_R$ are $y_n=n\pi (1+i),n\geq 1,|y_n|<R$. Hence, by \eqref{eq-sinh-1}, we have
\begin{align*}
&G(z)\mathop  = \limits^{z \to y_n} \frac{z^a}{4{\sinh^2 \left(\frac{i-1}{2}z \right)}}\left\{\frac4{(1+i)^2} \frac{1}{(z-y_n)^2}-\frac1{3}+o(1)\right\}.
\end{align*}
Further, the residue is
\begin{align}\label{residue-case-g-1}
\Res[G(z),y_n]&=\lim_{z\rightarrow y_n}\frac{d}{dz} \left\{(z-y_n)^2 \frac{z^a}{(\cos z-\cosh z)^2}\right\}\nonumber\\
&=\lim_{z\rightarrow y_n}\frac{d}{dz} \frac{z^a}{4{\sinh^2 \left(\frac{i-1}{2}z \right)}}\left\{\frac4{(1+i)^2} -\frac1{3}(z-y_n)^2+o((z-y_n)^2)\right\}\nonumber\\
&=a(1+i)^{a-3}\frac{(n\pi)^{a-1}}{\sinh^2(n\pi)}+(i-1)(1+i)^{a-2}\frac{(n\pi)^a\cosh(n\pi)}{\sinh^3(n\pi)}.
\end{align}
The integral over $\Gamma_R$ tends to $0$ as $R\rightarrow\infty$. Hence, by \eqref{residue-case-g-1} and the residue theorem, we find that
\begin{align}\label{residue-case-g-2}
(1-i^{a+1})\int_{0}^\infty \frac{z^adz}{(\cos z-\cosh z)^2}=2\pi i \sum_{n=1}^\infty \Res[G(z),y_n].
\end{align}
Thus, we have finished the proof of the theorem.
\end{proof}

Setting $a=7$ yields
\begin{align}
\sum_{n=1}^\infty \frac{n^7 \cosh(n\pi)}{\sinh^3(n\pi)}=\frac{7}{2\pi} \sum_{n=1}^\infty \frac{n^6}{\sinh^2(n\pi)}.
\end{align}
From \cite{XuZhao-2022}, we find that
\begin{align*}
\sum_{n=1}^\infty \frac{n^6}{\sinh^2(n\pi)}=\frac{\Gamma^{16}(1/4)}{7\cdot 2^{14}\pi^{12}}.
\end{align*}
Hence,
\begin{align}
\sum_{n=1}^\infty \frac{n^7 \cosh(n\pi)}{\sinh^3(n\pi)}=\frac{\Gamma^{16}(1/4)}{2^{15}\pi^{13}}.
\end{align}

In general, setting $a=4p+1$ and $4p-1$ in Theorems \ref{thm-Integ-series-1} and \ref{thm-Integ-series-2}, we get the following corollaries.
\begin{cor}\label{cor-Int-cos-h-1} For any integer $p\geq 0$,
\begin{align*}
\frac{(-1)^p2^{2p+1}}{\pi^{4p+1}}\int_0^\infty \frac{x^{4p+1}dx}{(\cos x+\cosh x)^2}
&=  \pi \sum_{n=1}^\infty \frac{(2n-1)^{4p+1}\sinh\left(\tn\pi\right)}{\cosh^3\left(\tn\pi\right)}
 - (4p+1)  \sum_{n=1}^\infty \frac{(2n-1)^{4p}}{\cosh^2\left(\tn\pi\right)}.
\end{align*}
\end{cor}

\begin{cor}\label{cor-Int-cos-h-2} For any integer $p\geq 1$,
\begin{align*}
 \frac{(-1)^p }{2^{2p-1}\pi^{4p+1}}\int_0^\infty \frac{x^{4p+1}dx}{(\cos x-\cosh x)^2}
&=(4p+1)\sum_{n=1}^\infty \frac{n^{4p}}{\sinh^2(n\pi)}
-2 \pi \sum_{n=1}^\infty \frac{n^{4p+1}\cosh(n\pi)}{\sinh^3(n\pi)}.
\end{align*}
\end{cor}

\begin{cor}\label{cor-Int-sin-h-1} For any positive integer $p$,
\begin{align}
\sum_{n=1}^\infty \frac{(2n-1)^{4p-1}\sinh\left(\tn\pi\right)}{\cosh^3\left(\tn\pi\right)}
=\frac{4p-1}{\pi}\sum_{n=1}^\infty \frac{(2n-1)^{4p-2}}{\cosh^2\left(\tn\pi\right)}.
\end{align}
\end{cor}

\begin{cor}\label{cor-Int-sin-h-2} For any positive integer $p\geq 2$,
\begin{align}
\sum_{n=1}^\infty \frac{n^{4p-1}\cosh(n\pi)}{\sinh^3(n\pi)}=\frac{4p-1}{2\pi} \sum_{n=1}^\infty \frac{n^{4p-2}}{\sinh^2(n\pi)}.
\end{align}
\end{cor}

\begin{re}
It is possible to establish the explicit relations between the Berndt-type integrals with $b\geq 3$ in \eqref{question-1}, \eqref{question-2} and some hyperbolic summations by using the method of contour integration.
\end{re}

\section{Berndt-type integrals}
In this section, we first provide a weaker version of Theorem~\ref{thm-main34} whose proof is much shorter.
\begin{thm}
For any positive integer $p$, the integrals
\begin{align}
\int_0^\infty \frac{x^{4p+1}}{(\cos x+\cosh x)^2}dx\quad \text{and}\quad \int_0^\infty \frac{x^{4p+1}}{(\cos x-\cosh x)^2}dx
\end{align}
can be evaluated in terms of $\Gamma(1/4)$ and $\pi$.
\end{thm}
\begin{proof}
From \cite[Thm. 3.6]{XuZhao-2022}, we know that the series
\begin{align*}
\sum_{n=1}^\infty \frac{n^{2p}}{\sinh^2(ny)}\quad \text{and}\quad \sum_{n=1}^\infty \frac{(2n-1)^{2p}}{\cosh^2((2n-1)y/2)}
\end{align*}
can be expressed in terms of $x,z$ and $z'$. Observing that
\begin{align*}
\sum_{n=1}^\infty \frac{n^{2p+1}\cosh(ny)}{\sinh^3(ny)}=-\frac12\frac{dx}{dy}\frac{d}{dx}\sum_{n=1}^\infty \frac{n^{2p}}{\sinh^2(ny)}
\end{align*}
and
\begin{align*}
\sum_{n=1}^\infty \frac{(2n-1)^{2p+1}\sinh((2n-1)y/2)}{\cosh^3((2n-1)y/2))}=-\frac{dx}{dy}\frac{d}{dx}\sum_{n=1}^\infty \frac{(2n-1)^{2p}}{\cosh^2((2n-1)y/2)}
\end{align*}
with
\begin{align}\label{Bruce-diff-xyz}
dx/dy=-x(1-x)z^2,
\end{align}
(see \cite[pp. 120, Entry 9(i)]{B1991}) we conclude that the two series
\begin{align*}
\sum_{n=1}^\infty \frac{n^{2p+1}\cosh(ny)}{\sinh^3(ny)}\quad\text{and}\quad \sum_{n=1}^\infty \frac{(2n-1)^{2p+1}\sinh((2n-1)y/2)}{\cosh^3((2n-1)y/2))}
\end{align*}
can be expressed in terms of a linear combinations of $x,z,z'$ and $z''$. Further, using \eqref{rel-den}, \eqref{den-z-diff} and taking $x=1/2,n=0,1,2$, we obtain that
\begin{align}\label{special-case-x}
y=\pi,z\Big(\frac 1{2}\Big)=\frac{\Gamma^2\Big(\frac 1{4}\Big)}{2\pi^{3/2}},\ z'\Big(\frac 1{2}\Big)=4\frac {\sqrt{\pi}}{\Gamma^2\Big(\frac 1{4}\Big)}\quad {\rm and}\quad z''\Big(\frac 1{2}\Big)=\frac{\Gamma^2\Big(\frac 1{4}\Big)}{2\pi^{3/2}},
\end{align}
where we have used the two classical relations
\[\Gamma(x+1)=x\Gamma(x)\quad{\rm and}\quad \Gamma(x)\Gamma(1-x)=\frac{\pi}{\sin(\pi x)}.\]
Hence, all of the four series
\begin{alignat*}{2}
& \sum_{n=1}^\infty \frac{n^{2p}}{\sinh^2(n\pi)}, && \sum_{n=1}^\infty \frac{(2n-1)^{2p}}{\cosh^2((2n-1)\pi/2)}, \\
& \sum_{n=1}^\infty \frac{n^{2p+1}\cosh(n\pi)}{\sinh^3(n\pi)},\quad && \sum_{n=1}^\infty \frac{(2n-1)^{2p+1}\sinh((2n-1)\pi/2)}{\cosh^3((2n-1)\pi/2))}
\end{alignat*}
can be evaluated by $\Gamma(1/4)$ and $\pi$. Finally, using Corollaries \ref{cor-Int-cos-h-1} and \ref{cor-Int-cos-h-2} with the help of the above, we obtain the desired description.
\end{proof}

From \cite[Cor. 3.7 and Cor. 3.8]{XuZhao-2022}, we can calculate some explicit formulas. For example, from \cite[Eqs. (3.73),(3.74) and (3.83)]{XuZhao-2022}, we have
\begin{align*}
&S_{2,2}(y)=\sum_{n=1}^\infty \frac{n^{2}}{\sinh^2(ny)}=\frac{x(1-x)z^2}{24}\big[z^2-4(1-2x)zz'-12x(1-x)(z')^2\big],\\
&S_{4,2}(y)=\sum_{n=1}^\infty \frac{n^{4}}{\sinh^2(ny)}=\frac{x(1-x)z^5}{120}\big[4(1-x+x^2)z'+(2x-1)z\big],\\
&C'_{2,2}(y)=\sum_{n=1}^\infty \frac{(2n-1)^{2}}{\cosh^2((2n-1)y/2)}=\frac{x(1-x)z^3}{3} \big[z-(1-2x)z'\big].
\end{align*}
Hence,
\begin{align*}
&\begin{aligned}
&\sum_{n=1}^\infty \frac{n^{3}\cosh(ny)}{\sinh^3(ny)}=-\frac1{2y'}\frac{d}{dx}\sum_{n=1}^\infty \frac{n^{2}}{\sinh^2(ny)}\\
&\qquad =\frac{(x^2-x)z^3}{48}\left\{\begin{array}{l} 24(x^2-x)^2(z')^3+4z^2\left[(x-3x^2+2x^3)z''+(7x^2-7x+1)z'\right]\\ +12(x^2-x)zz'\left[2(x^2-x)z''+(6x-3)z'\right]+(2x-1)z^3 \end{array}\right\},
\end{aligned}\\
&\begin{aligned}
&\sum_{n=1}^\infty \frac{n^{5}\cosh(ny)}{\sinh^3(ny)}=-\frac1{2y'}\frac{d}{dx}\sum_{n=1}^\infty \frac{n^{4}}{\sinh^2(ny)}\\
&\qquad=\frac{(x-x^2)z^6}{240}\left\{\begin{array}{l} 4(x-2x^2+2x^3-x^4)zz''+2(2-11x+21x^2-14x^3)zz'\\ +20(x-2x^2+2x^3-x^4)z'+(6x-6x^2-1)z^2 \end{array}\right\},
\end{aligned}\\
&\begin{aligned}
&\sum_{n=1}^\infty \frac{(2n-1)^{3}\sinh((2n-1)y/2)}{\cosh^3((2n-1)y/2))}=-\frac1{y'}\frac{d}{dx}\sum_{n=1}^\infty \frac{(2n-1)^{2}}{\cosh^2((2n-1)y/2)}\\
&\qquad=\frac{(x-x^2)z^4}{3}\left\{\begin{array}{l} 3(3x^2-x-2x^3)(z')^2+(1-2x)z^2\\+z\left[(3x^2-2x^3-x)z''+(10x-10x^2-1)z'\right]\end{array}\right\}.
\end{aligned}
\end{align*}
Setting $x=1/2$ yields
\begin{align}
&\sum_{n=1}^\infty \frac{n^{3}\cosh(n\pi)}{\sinh^3(n\pi)}=\frac{\Gamma^8(1/4)}{2^{10}\pi^7}-\frac1{16\pi^3}, \label{equ:chsh3specialCasem=1}\\
&\sum_{n=1}^\infty \frac{n^{5}\cosh(n\pi)}{\sinh^3(n\pi)}=\frac{\Gamma^{16}(1/4)}{3\cdot 2^{16}\pi^{12}} +\frac{\Gamma^8(1/4)}{2^{10}\pi^8}, \label{equ:chsh3Casem=2}\\
&\sum_{n=1}^\infty \frac{(2n-1)^{3}\sinh((2n-1)\pi/2)}{\cosh^3((2n-1)\pi/2))}=\frac{\Gamma^8(1/4)}{64\pi^7}.\nonumber
\end{align}

We can also obtain more similar evaluations with higher powers of $n$ or $2n-1$ on the numerators.

\section{Proof of Theorem \ref{thm-main1} and first part of Theorem \ref{thm-main34}}
We now apply the results from previous sections to prove Theorem \ref{thm-main1}. The crucial ones are the
properties of the $q$-series $\Phi_{1,m}$ listed in Prop.~\ref{pro:RamanujanPhi1}.

\begin{thm} \label{thm-main1Precise}
For any positive integer $m$,
\begin{align}
\sum_{n=1}^\infty \frac{n^{4m-2}}{\sinh^2(n\pi)}  &\,\in \frac{\Gamma^{8m}(1/4)}{\pi^{6m}}\Q+\frac{\delta_{m,1}}{\pi^2}\Q, \label{equ:sh2odd}\\
\sum_{n=1}^\infty \frac{n^{4m-4}}{\sinh^2(n\pi)}  &\,\in \frac{\Gamma^{8m-8}(1/4)}{\pi^{6m-5}}\Q+\delta_{m,1}\Q, \label{equ:sh2even}\\
\sum_{n=1}^\infty \frac{n^{4m-1}\cosh(n\pi)}{\sinh^3(n\pi)}  &\,\in \frac{\Gamma^{8m}(1/4)}{\pi^{6m+1}}\Q +\frac{\delta_{m,1}}{\pi^3}\Q, \label{equ:chsh3odd}\\
\sum_{n=1}^\infty \frac{n^{4m-3}\cosh(n\pi)}{\sinh^3(n\pi)}  &\,\in \frac{\Gamma^{8m}(1/4)}{\pi^{6m}}\Q
+\frac{\Gamma^{8m-8}(1/4)}{\pi^{6m-4}}\Q. \label{equ:chsh3even}
\end{align}
\end{thm}

\begin{proof} From the last equation of \cite[Table 1(i)]{Z1979}, setting $s=1$ and replacing $\pi c$ by $y$ we get
\begin{align}\label{Eq-y-sinh}
\sum_{n=1}^\infty \frac{1}{\sinh^2(ny)}=4\sum_{n=1}^\infty \frac{n}{e^{2ny}-1}.
\end{align}
Setting $y=\pi$ yields the well-known identity (see \cite[Eq. (20)]{Z1979})
\begin{align}
\sum_{n=1}^\infty \frac{1}{\sinh^2(n\pi)}=\frac1{6}-\frac{1}{2\pi}.
\end{align}
This settles the $m=1$ case of \eqref{equ:sh2even}. The $m=1$ case of \eqref{equ:sh2odd}
is given by the first entry in \cite[Table 4.1]{XuZhao-2022} (denoted by $S_{2,2}(\pi)$ in \cite[Conjecture 4.1]{XuZhao-2022}). We
will also calculate it by \eqref{equ:provesh2oddm=1} in Examples \ref{exa:sinhEg}.
The special case of \eqref{equ:chsh3odd} when $m=1$ is given by \eqref{equ:chsh3specialCasem=1}.
This can also be proved by using \eqref{equ:Phi'12} as is done in Examples \ref{exa:sinhEg}.

As the last special case we also need to prove \eqref{equ:chsh3even} when $m=1$.
Differentiating \eqref{Eq-y-sinh} with respect to $y$ we have
\begin{align}\label{Eq-d-y-sinh}
\sum_{n=1}^\infty \frac{n\cosh(ny)}{\sinh^3(ny)}=4\sum_{n=1}^\infty \frac{n^2 e^{2ny}}{(e^{2ny}-1)^2}=\sum_{n=1}^\infty \frac{n^2}{\sinh^2(ny)}.
\end{align}
Hence, setting $y=\pi$ and using the $m=1$ case of \eqref{equ:sh2odd} just shown above we get
\begin{align*}
\sum_{n=1}^\infty \frac{n\cosh(n\pi)}{\sinh^3(n\pi)}=  \frac{\Gamma^8(1/4)}{3\cdot 2^9\pi^6}-\frac1{8\pi^2}.
\end{align*}

For general cases, by Prop.~\ref{pro:RamanujanPhi1} we see that for all $m\ge1$
\begin{align*}
\sum_{n=1}^\infty \frac{n^{4m+2}}{\sinh^2(ny)}  &\,=4 \Phi_{1,4m+2}\in z^{4m+4}\Q[x],\\
\sum_{n=1}^\infty \frac{n^{4m}}{\sinh^2(ny)}  &\,=4\Phi_{1,4m}\in z^{4m+1}z'\Q[x],\\
\sum_{n=1}^\infty \frac{n^{4m+3}\cosh(ny)}{\sinh^3(n\pi)}  &\,
   =2 x(1-x)z^2 \Phi'_{1,4m+2}\in z^{4m+5}z'\Q[x], \\
\sum_{n=1}^\infty \frac{n^{4m+1}\cosh(ny)}{\sinh^3(ny)}  &\,=2 x(1-x)z^2 \Phi'_{1,4m}  \in z^{4m+4}\Q[x]+z^{4m+2}(z')^2\Q[x]+z^{4m+3}z''\Q[x],
\end{align*}
by \eqref{Bruce-diff-xyz}. Thus the theorem follows immediately by taking $x=1/2$ and using \eqref{special-case-x}.
\end{proof}

\begin{exa} \label{exa:sinhEg}
According to the theorem for all $m\ge 1$ we can put
\begin{align*}
\sum_{n=1}^\infty \frac{n^{4m-2}}{\sinh^2(n\pi)}= &\,\alpha_{4m-2}\frac{\Gamma^{8m}(1/4)}{\pi^{6m}}-\frac{\delta_{m,1}}{8\pi^2},\\
\sum_{n=1}^\infty \frac{n^{4m}}{\sinh^2(n\pi)}=&\,\alpha_{4m}\frac{\Gamma^{8m}(1/4)}{\pi^{6m+1}}, \\
\sum_{n=1}^\infty \frac{n^{4m-1}\cosh(n\pi)}{\sinh^3(n\pi)} = &\,\beta_{4m-1}\frac{\Gamma^{8m}(1/4)}{\pi^{6m+1}}+\frac{\delta_{m,1}\gamma_3}{\pi^3}, \\
\sum_{n=1}^\infty \frac{n^{4m+1}\cosh(n\pi)}{\sinh^3(n\pi)} = &\,\beta_{4m+1}\frac{\Gamma^{8m+8}(1/4)}{\pi^{6m+6}}  +\gamma_{4m+1}\frac{\Gamma^{8m}(1/4)}{\pi^{6m+2}}.
\end{align*}
The following table provides the explicit expression of $\Phi_{1,2m}$ modulo $R\Q[P,Q,R]$
for small values of $m$, extending those contained in \cite[Table II]{Rama1916}.
Using the values of $P,Q$ in \eqref{3.583} and \eqref{3.55} and setting $x=1/2$ we get the third and last row.
\begin{center}
\begin{tabular}{|c|c|c|c|c|c|c|c|c|c|c| } \hline
 $k$         &  2 & 4  & 6   & 8 &  10  & 12 & 14   & 16 & 18  & 20    \\ \hline
$\phantom{\frac11}\hskip-4pt 4\Phi_{1,k}$ & $\tfrac{Q-P^2}{72}$ & $\tfrac{PQ}{180}$ & $\tfrac{Q^2}{252}$ & $\tfrac{PQ^2}{180}$ & $\tfrac{Q^3}{132}$ & $\tfrac{7PQ^3}{260}$ & $\tfrac{Q^4}{12}$ &  $\tfrac{539PQ^4}{1020}$ & $\tfrac{203Q^5}{76}$ & $\tfrac{539PQ^5}{20}$   \\  \hline
 $\phantom{\frac11}\hskip-4pt x=\tfrac12$ & $\tfrac{z^4 - 3z^2(z')^2}{96} $ & $\tfrac{z^5z'}{160}$ & $\tfrac{z^8}{448}$ & $\tfrac{3z^9z'}{640}$ & $\tfrac{9z^{12}}{2816}$ & $\tfrac{567z^{13}z'}{33280}$ & $\tfrac{27z^{16}}{1024}$ & $\tfrac{43659z^{17}z'}{174080}$ & $\tfrac{49329z^{20}}{77824}$ & $\tfrac{392931z^{21}z'}{40960}$\\  \hline
$\phantom{\frac11}\hskip-4pt \alpha_{k}$ & $\tfrac{1}{3\cdot2^9} $ & $\tfrac{1}{5\cdot2^8}$ & $\tfrac{1}{7\cdot2^{14}}$ & $\tfrac{3}{5\cdot2^{14}}$ & $\tfrac{3^2}{11\cdot2^{20}}$ & $\tfrac{7\cdot3^4}{5\cdot13\cdot2^{20}}$ & $\tfrac{3^3}{2^{26}}$ & $\tfrac{11\cdot3^4\cdot7^2}{5\cdot17\cdot2^{26}}$ & $\tfrac{7\cdot29\cdot3^5}{19\cdot2^{32}}$ & $\tfrac{11\cdot3^6\cdot7^2}{5\cdot2^{32}}$\\  \hline
\end{tabular}
\end{center}

In particular, if $k=2$ then by setting $x=1/2$ we see that
\begin{equation}\label{equ:provesh2oddm=1}
    4\Phi_{1,2}= \frac{\Gamma^8(1/4)}{3\cdot2^9} - \frac{1}{8\pi^2}.
\end{equation}
This proves the $m=1$ case of \eqref{equ:sh2odd}.
We may write down the $k=4,6,8,10$ cases explicitly below:
\begin{alignat*}{3}
\sum_{n=1}^\infty \frac{n^{4}}{\sinh^2(n\pi)}=&\,\frac{\Gamma^{8}(1/4)}{5\cdot2^8\pi^{7}}, \quad &
\sum_{n=1}^\infty \frac{n^{6}}{\sinh^2(n\pi)}= &\,\frac{\Gamma^{16}(1/4)}{7\cdot2^{14}\pi^{12}},\\
\sum_{n=1}^\infty \frac{n^{8}}{\sinh^2(n\pi)}=&\,\frac{3\Gamma^{16}(1/4)}{5\cdot2^{14}\pi^{13}}, \quad &
\sum_{n=1}^\infty \frac{n^{10}}{\sinh^2(n\pi)}= &\,\frac{9\Gamma^{24}(1/4)}{11\cdot2^{20}\pi^{18}}.
\end{alignat*}

Similarly, setting $x=1/2$ in \eqref{equ:Phi'4m} and \eqref{equ:Phi'4m+2}
we arrive at the following table.
\begin{center}
\begin{tabular}{|c|c|c|c|c|c|c|c|c|c|c|c| } \hline
 $k$      &  3 & 5 & 7  & 9 &  11  & 13 & 15   & 17 & 19  & 21 & 23   \\ \hline
$\phantom{\frac11}\hskip-4pt \beta_{k}$ & $\tfrac{1}{2^{10}}$ & $\tfrac{1}{3\cdot 2^{16}}$ & $\tfrac{ 1}{2^{15}}$ & $\tfrac{ 1}{2^{22}}$ & $\tfrac{ 3^2}{2^{21}}$ & $\tfrac{ 19\cdot 3^2}{7\cdot 2^{28}}$ & $\tfrac{ 5\cdot 3^4}{2^{27}}$ & $\tfrac{ 67\cdot 3^3}{2^{34}}$ & $\tfrac{ 7\cdot 3^5\cdot 29}{2^{33}}$ & $\tfrac{ 15629\cdot 3^5}{11\cdot 2^{40}}$ & $\tfrac{ 389\cdot 3^6\cdot 7^2}{2^{39}}$
 \\ \hline
$\phantom{\frac11}\hskip-4pt\gamma_{k}$ & $-\tfrac{1}{2^{4}}$ & $\tfrac{1}{2^{10}}$ & $0$ & $\tfrac{ 3^3}{5\cdot 2^{16}}$ & $0$ & $\tfrac{ 7\cdot 3^4}{5\cdot 2^{22}}$ & $0$ & $\tfrac{ 11\cdot 7^2\cdot 3^4}{5\cdot 2^{28}}$ & $0$ & $\tfrac{ 11\cdot 7^3\cdot 3^7}{5\cdot 2^{34}}$ & $0$
 \\ \hline
\end{tabular}
\end{center}

In particular,  $k=3$ implies the special case $m=1$ in \eqref{equ:chsh3odd}.
We may write down the $k=5,7,9,11$ cases explicitly below.
\begin{alignat*}{3}
&\sum_{n=1}^\infty \frac{n^{5}\cosh(n\pi)}{\sinh^3(n\pi)}=\frac{\Gamma^{16}(1/4)}{3\cdot 2^{16}\pi^{12}}+\frac{\Gamma^8(1/4)}{2^{10}\pi^8}, \quad&
&\sum_{n=1}^\infty \frac{n^{7}\cosh(n\pi)}{\sinh^3(n\pi)}=\frac{\Gamma^{16}(1/4)}{2^{15}\pi^{13}},\\
&\sum_{n=1}^\infty \frac{n^{9}\cosh(n\pi)}{\sinh^3(n\pi)}=\frac{\Gamma^{24}(1/4)}{2^{22}\pi^{18}}+\frac{27\Gamma^{16}(1/4)}{5\cdot 2^{16}\pi^{14}}, \quad &
&\sum_{n=1}^\infty \frac{n^{11}\cosh(n\pi)}{\sinh^3(n\pi)}=\frac{9\Gamma^{24}(1/4)}{2^{21}\pi^{19}}.
\end{alignat*}
Note that the $k=5$ case agrees with \eqref{equ:chsh3Casem=2}.
\end{exa}

We now prove the evaluation of the first integral in Theorem \ref{thm-main34}.

\begin{thm}  \label{thm-main3Precise}
For any positive integer $p$,
\begin{align}\label{equ:conj-for-a2}
&\int_0^\infty \frac{x^{4p+1}}{(\cos x-\cosh x)^2}dx=i_p\frac{\Gamma^{8p}(1/4)}{\pi^{2p}}+j_p\frac{\Gamma^{8p+8}(1/4)}{\pi^{2p+4}}
\end{align}
for some $i_p,j_p\in\Q$.
\end{thm}
\begin{proof}
This follows from Cor. \ref{cor-Int-cos-h-2} and Theorem \ref{thm-main1Precise} easily.
\end{proof}

\begin{exa} \label{exa-Integral-}
By computing more cases in Theorem \ref{thm-main1Precise} we can obtain the following evaluations:
\begin{align*}
&\int_0^\infty \frac{x^{5}}{(\cos x-\cosh x)^2}dx=-\frac{\Gamma^{8}(1/4)}{2^{8}\pi^{2}}+\frac{\Gamma^{16}(1/4)}{3\cdot 2^{14}\pi^{6}},\\
&\int_0^\infty \frac{x^{9}}{(\cos x-\cosh x)^2}dx=\frac{27\Gamma^{16}(1/4)}{5\cdot 2^{12}\pi^{4}}-\frac{\Gamma^{24}(1/4)}{2^{18}\pi^{8}},\\
&\int_0^\infty \frac{x^{13}}{(\cos x-\cosh x)^2}dx=-\frac{567\Gamma^{24}(1/4)}{5\cdot 2^{16}\pi^{6}}+\frac{171\Gamma^{32}(1/4)}{7\cdot 2^{22}\pi^{10}},\\
&\int_0^\infty \frac{x^{17}}{(\cos x-\cosh x)^2}dx=\frac{43659 \Gamma^{32}(1/4)}{5\cdot 2^{20}\pi^{8}}-\frac{1809\Gamma^{40}(1/4)}{2^{26}\pi^{12}},\\
&\int_0^\infty \frac{x^{21}}{(\cos x-\cosh x)^2}dx=-\frac{8251551 \Gamma^{40}(1/4)}{5\cdot  2^{24} \pi ^{10}}+\frac{3797847 \Gamma^{48}(1/4)}{11\cdot  2^{30} \pi ^{14}},\\
&\int_0^\infty \frac{x^{25}}{(\cos x-\cosh x)^2}dx=\frac{8622870795\Gamma^{48}(1/4)}{13\cdot  2^{28}\pi ^{12}}-\frac{138429081 \Gamma^{56}(1/4)}{2^{34} \pi ^{16}},\\
&\int_0^\infty \frac{x^{29}}{(\cos x-\cosh x)^2}dx=-\frac{2498907956391 \Gamma^{56}(1/4)}{5\cdot 2^{32}\pi ^{14}}+\frac{104367224493\Gamma^{64}(1/4)}{2^{38}\pi^{18}}.
\end{align*}
\end{exa}

\section{Some preliminary results on Jacobi elliptic functions}\label{sec-Jacobi Functions}
In order to prove Theorem \ref{thm-main2} and the second part of Theorem \ref{thm-main34} we need to switch gears and
turn to Jacobi elliptic functions. This large family of functions have a very long history and
many books have been devoted to its study. We will only need some of them in this paper.

The Jacobi elliptic function $\s (u)$ is defined via the inversion of the elliptic integral
\begin{align}
u=\int_0^\varphi \frac{dt}{\sqrt{1-k^2\sin^2 t}}\quad (0<k^2<1),
\end{align}
namely, $\s (u):=\sin \varphi$, where $k=\mod u$ is referred to as the elliptic modulus, and $\varphi={\rm am}(u,k)={\rm am}(u)$ is referred to as the Jacobi amplitude. The Jacobi elliptic functions $\cn u$ and $\dn u$ may be defined as follows:
\begin{align*}
&\cn (u):=\sqrt{1-\s^2 (u)}, \quad  \dn (u):=\sqrt{1-k^2\s ^2(u)}.
\end{align*}
Using the Jacobi elliptic functions $\s (u)$, $\cn (u)$ and $\dn (u)$, we give the definitions of Jacobi elliptic functions $\nc (u)$, $\dc (u)$ and $\ds (u)$, as follows:
\begin{align*}
&\nc (u):=\frac{1}{\cn (u)}, \quad \dc (u):=\frac{\dn (u)}{\cn (u)}, \quad \ds (u):=\frac{\dn (u)}{\s (u)}.
\end{align*}

We will need to use the Maclaurin (or Laurent) series of these functions. First, we consider their higher derivatives.

\begin{lem}\label{lem:nc}
For any integer $m\ge 0$ there is a polynomial $f_m(X,Y,Z)\in\Z[X,Y,Z]$ such that
\begin{equation*}
\frac{d^m}{du^m} \nc(u)=f_m(\sin \gf, \cos\gf,k^2)(1-\sin^2\gf)^{-m-1/2}\cdot
\left\{
  \begin{array}{ll}
    (1-k^2\sin^2\gf)^{1/2}, \quad\ & \hbox{if $m$ is odd;} \\
    1, & \hbox{if $m$ is even.}
  \end{array}
\right.
\end{equation*}
\end{lem}
\begin{proof} We proceed by induction on $m$. By definition
\begin{equation*}
\nc(u)=(1-\sin^2\gf)^{-1/2}
\end{equation*}
so we get $f_0(X,Y,Z)=1$. By the chain rule
\begin{equation*}
\frac{d}{du} \nc(u)=\sin \gf\cos \gf (1-\sin^2\gf)^{-3/2}\frac{d\gf}{du}
    =\sin \gf\cos \gf (1-\sin^2\gf)^{-3/2} (1-k^2\sin^2\gf)^{1/2}.
\end{equation*}
Thus $f_1(X,Y,Z)=XY$. Suppose the lemma holds for $m$, namely, $f_m(X,Y,Z)\in\Z[X,Y,Z]$ for some
$m\ge 1$. If $m+1>1$ is even then
\begin{align*}
\frac{d^{m+1}}{du^{m+1}} \nc(u) =&\,\frac{d}{d\gf}\Big(f_m(\sin \gf, \cos\gf,k^2)
    (1-\sin^2\gf)^{-m-1/2} (1-k^2\sin^2\gf)^{1/2}\Big)\cdot (1-k^2\sin^2\gf)^{1/2} \\
=&\, \frac{d}{d\gf}\Big(f_m(\sin \gf, \cos\gf,k^2)\Big)(1-\sin^2\gf)^{-m-1/2}(1-k^2\sin^2\gf)  \\
&\, +(2m+1) f_m(\sin \gf, \cos\gf,k^2)\sin\gf\cos\gf(1-\sin^2\gf)^{-m-3/2} (1-k^2\sin^2\gf) \\
&\, -f_m(\sin \gf, \cos\gf,k^2)(1-\sin^2\gf)^{-m-1/2} (k^2\sin \gf\cos\gf).
\end{align*}
Hence
\begin{align*}
f_{m+1}(X,Y,Z)=&\,\left(Y\frac{\partial}{\partial X}f_m(X,Y,Z)-X\frac{\partial}{\partial Y}f_m(X,Y,Z)\right)(1-X^2)(1-X^2Z) \\
&\,+XY\Big((2m+1)(1-X^2Z) - Z(1-X^2)\Big) f_m(X,Y,Z)\in\Z[X,Y,Z]
\end{align*}
by induction assumption. If $m+1>1$ is odd then
\begin{align*}
\frac{d^{m+1}}{du^{m+1}} \nc(u) =&\,\frac{d}{d\gf}\Big(f_m(\sin \gf, \cos\gf,k^2)
    (1-\sin^2\gf)^{-m-1/2} \Big)\cdot (1-k^2\sin^2\gf)^{1/2} \\
=&\, \left\{\frac{d}{d\gf}\Big(f_m(\sin \gf, \cos\gf,k^2)\Big)(1-\sin^2\gf)^{-m-1/2} \right.  \\
&\,\left. +(2m+1) f_m(\sin \gf, \cos\gf,k^2)\sin\gf\cos\gf(1-\sin^2\gf)^{-m-3/2} \right\}(1-k^2\sin^2\gf)^{1/2}
\end{align*}
Hence
\begin{align*}
f_{m+1}(X,Y,Z)=&\,\left(Y\frac{\partial}{\partial X}f_m(X,Y,Z)-X\frac{\partial}{\partial Y}f_m(X,Y,Z)\right)(1-X^2) \\
&\, +(2m+1)XY f_m(X,Y,Z) \in\Z[X,Y,Z]
\end{align*}
by induction assumption.  This completes the proof of the lemma.
\end{proof}

By Lemma \ref{lem:nc} we obtain the following corollary at once.

\begin{cor}\label{cor-nc}
The Maclaurin series expansion of $\nc(u)$ has the form
\begin{equation}\label{equ:cor-nc}
    \sum_{m\ge 0} \frac{f_m(x)}{m!}u^m, \quad\text{where } f_m(x)\in\Z[x] \ \forall m\ge 0.
\end{equation}
\end{cor}

The proof of the following lemma is completely similar to the above.
\begin{lem}
For any integer $m\ge 0$ there is a polynomial $g_m(X,Y,Z), h_m(X,Y,Z), l_m(X,Y,Z)\in\Z[X,Y,Z]$ such that
\begin{align*}
\frac{d^m}{du^m} \s (u)=&\, g_m(\sin \gf, \cos\gf,k^2)\cdot
\left\{
  \begin{array}{ll}
    (1-k^2\sin^2\gf)^{1/2}, \quad\ & \hbox{if $m$ is odd;} \\
    1, & \hbox{if $m$ is even,}
  \end{array}
\right. \\
\frac{d^m}{du^m} \cn(u)=&\, h_m(\sin \gf, \cos\gf,k^2)(1-\sin^2\gf)^{-m+1/2}\cdot
\left\{
  \begin{array}{ll}
    1, & \hbox{if $m$ is odd;} \\
    (1-k^2\sin^2\gf)^{1/2}, \quad\ & \hbox{if $m$ is even.}
  \end{array}
\right.\\
\frac{d^m}{du^m} \dn(u)=&\, l_m(\sin \gf, \cos\gf,k^2)\cdot
\left\{
  \begin{array}{ll}
    1, & \hbox{if $m$ is odd;} \\
    (1-k^2\sin^2\gf)^{1/2}, \quad\ & \hbox{if $m$ is even.}
  \end{array}
\right.
\end{align*}
\end{lem}

Using the fact the Maclaurin series of $\s (u)$ starts with $u+O(u^2)$ we can prove
the following results immediately.

\begin{cor}\label{cor-dc+ds}
There are polynomials where $p_m(x),q_m(x)\in\Q[x]$ such that
the Maclaurin series of $\dc(u)$ has the form
\begin{equation} \label{equ:cor-dc-fe-u}
    \sum_{m\ge 0} \frac{p_m(x)}{m!} u^m
\end{equation}
and the Maclaurin series of $u\cdot \ds(u)$ has the form
\begin{equation} \label{equ:cor-ds-fe-u}
    \sum_{m\ge 0} \frac{q_m(x)}{m!} u^m.
\end{equation}
\end{cor}

\begin{re}
See \cite{MMYZ2021} and the references therein for recent results on Jacobian elliptic functions.
\end{re}

\section{Evaluations of hyperbolic summations via Jacobi functions}
In this section, we will apply the Fourier series expansions and power series expansions of Jacobi elliptic functions ${\rm dc},{\rm nc}$ and $\ds$ to establish some explicit evaluations of hyperbolic summations.

To begin with, we record the Fourier series expansions of Jacobi elliptic functions $\nc (u)$, $\dc (u)$ and $\ds (u)$
\begin{align}
&\dc(u)=\frac{\pi}{2K}\sec\Big(\frac{\pi u}{2K}\Big)+\frac{2\pi}{K} \sum_{n=0}^\infty \frac{(-1)^n q^{2n+1}\cos\Big((2n+1)\frac{\pi u}{2K}\Big)}{1-q^{2n+1}},\label{fe-dc}\\
&\nc(u)=\frac{\pi}{2K\sqrt{1-k^2}}\sec\Big(\frac{\pi u}{2K}\Big)-\frac{2\pi}{K\sqrt{1-k^2}} \sum_{n=0}^\infty \frac{(-1)^n q^{2n+1}\cos\Big((2n+1)\frac{\pi u}{2K}\Big)}{1+q^{2n+1}},\label{fe-nc}\\
&\ds(u)=\frac{\pi}{2K}\csc\Big(\frac{\pi u}{2K}\Big)-\frac{2\pi}{K} \sum_{n=0}^\infty \frac{ q^{2n+1}\sin\Big((2n+1)\frac{\pi u}{2K}\Big)}{1+q^{2n+1}},\label{fe-ds}
\end{align}
given in \cite{DCLMRT1992} and in the classical text \cite[pp. 511-512]{WW1966}.

Recall that the Euler numbers $E_m$ are defined by the generating function
\begin{equation*}
\sec(x)=\sum_{m=0}^\infty \frac{E_m}{(2m)!} x^{2m}.
\end{equation*}
In particular, $E_0=1,E_1=1,E_2=5,E_3=61,E_4=1385,\ldots$.

\begin{thm}\label{thm-dc-cases}
Let $x,y$ and $z$ be defined by \eqref{rel-den}. For any integer $m\ge 0$,
\begin{align}\label{equ:thm-dc-cases}
\sum_{n=0}^\infty \frac{(-1)^n(2n+1)^{2m}}{e^{(2n+1)y}-1}
=\frac{(-1)^m}4\Big(z^{2m+1} p_{2m}(x)-E_m\Big)  \in \mathbb{Q}[x,z],
\end{align}
where $p_n(x)$ appears in the coefficients of the Maclaurin series of $\dc(u)$ in \eqref{equ:cor-dc-fe-u}.
\end{thm}
\begin{proof} On the one hand, using the notations of \eqref{rel-den}, the \eqref{fe-dc} can be rewritten as
\begin{align}\label{dc-fe-psx}
\dc(u)&=\frac1{z}\sec\Big(\frac{u}{z}\Big)+\frac{4}{z}\sum_{n=0}^\infty \frac{(-1)^ne^{-(2n+1)y}\cos\Big((2n+1)\frac{u}{z}\Big)}{1-e^{-(2n+1)y}}\nonumber\\
&=\frac1{z}\sum_{m=0}^\infty \frac{E_m}{(2m)!} \Big(\frac{u}{z}\Big)^{2m}+\frac{4}{z}\sum_{n=0}^\infty \frac{(-1)^n}{e^{(2n+1)y}-1}\sum_{m=0}^\infty \frac{(-1)^m}{(2m!)}\Big((2n+1) \frac{u}{z}\Big)^{2m}\nonumber\\
&=\sum_{m=0}^\infty\left\{E_m+4 (-1)^m \sum_{n=0}^\infty \frac{(-1)^n (2n+1)^{2m}}{e^{(2n+1)y}-1} \right\} \frac{u^{2m}}{(2m)!z^{2m+1}}.
\end{align}
Thus, comparing the coefficients of $u^{2m}$ in \eqref{dc-fe-psx} and \eqref{equ:cor-dc-fe-u} yields \eqref{equ:thm-dc-cases} quickly.
\end{proof}

\begin{exa}
With the help of \emph{Mathematica}, we calculate the power series expansion of Jacobi elliptic function $\dc(u)$ as
\begin{align*}
\dc(u)&=1+\frac1{2}(1-x)u^2+\frac{1}{4!}(5-6x+x^2)u^4+\frac{1}{6!}(61-107x+47x^2-x^3)u^6\nonumber\\
&\quad+\frac1{8!}(1385-3116x+2142x^2-412x^3+x^4)u^8\nonumber\\
&\quad+\frac1{10!}(50521-138933x+130250x^2-45530x^3+3693x^4-x^5)u^{10}+\cdots.
\end{align*}
Comparing the coefficients of $u^0,u^2,u^4$ and $u^6$, we see that
\begin{align}
&\sum_{n=0}^\infty \frac{(-1)^n}{e^{(2n+1)y}-1}=\frac{z}{4}-\frac1{4},\label{eq-dc-z-1}\\
&\sum_{n=0}^\infty \frac{(-1)^n(2n+1)^2}{e^{(2n+1)y}-1}=\frac1{4}-\frac1{4}z^3(1-x),\label{eq-dc-z-2}\\
&\sum_{n=0}^\infty \frac{(-1)^n(2n+1)^4}{e^{(2n+1)y}-1}=-\frac5{4}+\frac1{4}z^5(5-6x+x^2),\label{eq-dc-z-3}\\
&\sum_{n=0}^\infty \frac{(-1)^n(2n+1)^6}{e^{(2n+1)y}-1}=\frac{61}{4}-\frac1{4}z^7(61-107x+47x^2-x^3).\label{eq-dc-z-4}
\end{align}
\end{exa}

\begin{re}
In his recent work \cite{Campbell} J.M. Campbell also applied the method based on Fourier series expansions related to Jacobi elliptic functions ${\rm dc}$ and ${\rm nc}$ to discover some new identities evaluating hyperbolic infinite sums in terms of the complete elliptic integrals $K$ and $E$.
\end{re}

\begin{thm}\label{thm-nc-cases}
Let $x,y$ and $z$ be defined by \eqref{rel-den}. For any positive integer $m\ge 0$,
\begin{align} \label{equ:nc-cases}
\sum_{n=0}^\infty \frac{(-1)^n(2n+1)^{2m}}{e^{(2n+1)y}+1}
=\frac{(-1)^m}4 \Big(E_m-\sqrt{1-x} z^{2m+1}f_{2m}(x)\Big)
\in \mathbb{Q}[\sqrt{1-x},x,z],
\end{align}
where $f_m(x)$ appears in the coefficients of the Maclaurin series of $\nc(u)$ in \eqref{cor-nc}.
\end{thm}

\begin{proof}
The proof is completely similar to that of Theorem \ref{thm-dc-cases}. On the one hand, applying \eqref{rel-den}, we obtain
\begin{align}\label{nc-fe-psx}
\nc(u)=\frac{1}{\sqrt{1-x}}\sum_{m=0}^\infty \left\{ E_m -4(-1)^m  \sum_{n=0}^\infty \frac{(-1)^n (2n+1)^{2m}}{e^{(2n+1)y}+1} \right\} \frac{u^{2m}}{(2m)!z^{2m+1}}.
\end{align}
Comparing the coefficients of $u^{2m}$ in \eqref{nc-fe-psx} and \eqref{equ:cor-nc}, we can derive \eqref{equ:nc-cases} easily.
\end{proof}

\begin{exa}  Using \emph{Mathematica}, we can find that
\begin{align*}
\nc(u)&=1+\frac1{2}u^2+\frac{1}{4!}(5-4x)u^4+\frac{1}{6!}(61-76x+16x^2)u^6
\nonumber\\&\quad+\frac1{8!}(1385-2424x+1104x^2-64x^3)u^8\nonumber\\
&\quad+\frac1{10!}(50521-113672x+79728x^2-16832x^3+256x^4)u^{10}+\cdots.
\end{align*}
Comparing the coefficients of $u^0,u^2,u^4$ and $u^6$, we deduce that
\begin{align}
&\sum_{n=0}^\infty \frac{(-1)^n}{e^{(2n+1)y}+1}=\frac{1}{4}-\frac1{4}\sqrt{1-x}z,\label{eq-nc-z-1}\\
&\sum_{n=0}^\infty \frac{(-1)^n(2n+1)^2}{e^{(2n+1)y}+1}=\frac1{4}\sqrt{1-x}z^3-\frac{1}{4},\label{eq-nc-z-2}\\
&\sum_{n=0}^\infty \frac{(-1)^n(2n+1)^4}{e^{(2n+1)y}+1}=\frac{5}{4}-\frac{1}{4}\sqrt{1-x}(5-4x)z^5,\label{eq-nc-z-3}\\
&\sum_{n=0}^\infty \frac{(-1)^n(2n+1)^6}{e^{(2n+1)y}+1}=-\frac{61}{4}+\frac{1}{4}\sqrt{1-x}(61-76x+16x^2)z^7,\label{eq-nc-z-4}.
\end{align}
\end{exa}

\begin{re}
The evaluations \eqref{eq-dc-z-1}-\eqref{eq-dc-z-4} can be found in \cite[Entry 17(vi)-(ix)]{B1991}. Those in \eqref{eq-dc-z-1}, \eqref{eq-dc-z-2} and \eqref{eq-nc-z-1} can also be found in \cite[Eq. (15)]{Campbell}, \cite[Thm. 1]{Campbell} and \cite[Eq. (35)]{Campbell}, respectively.
\end{re}

Utilizing \eqref{Bruce-diff-xyz} and the fact that ($m\ge 0$)
\begin{align*}
&\frac{d}{dx}\sum_{n=0}^\infty \frac{(-1)^n(2n+1)^{2m}}{e^{(2n+1)y}-1}=-\frac1{4}\frac{dy}{dx} \cdot  \sum_{n=0}^\infty \frac{(-1)^n(2n+1)^{2m+1}}{\sinh^2\Big(\frac{2n+1}{2}y\Big)},\\
&\frac{d}{dx}\sum_{n=0}^\infty \frac{(-1)^n(2n+1)^{2m}}{e^{(2n+1)y}+1}=-\frac1{4}\frac{dy}{dx} \cdot  \sum_{n=0}^\infty \frac{(-1)^n(2n+1)^{2m+1}}{\cosh^2\Big(\frac{2n+1}{2}y\Big)},
\end{align*}
we obtain the following corollary immediately from Theorems \ref{thm-dc-cases} and \ref{thm-nc-cases}.

\begin{cor} For any integer $m\ge 0$,
\begin{align*}
&\sum_{n=0}^\infty \frac{(-1)^n(2n+1)^{2m+1}}{\sinh^2\Big(\frac{2n+1}{2}y\Big)}
= (-1)^m x(1-x)z^2 \frac{d}{dx} \Big(z^{2m+1} p_{2m}(x)\Big)\in \mathbb{Q}[x,z,z'],\\
&\sum_{n=0}^\infty \frac{(-1)^n(2n+1)^{2m+1}}{\cosh^2\Big(\frac{2n+1}{2}y\Big)}
=- (-1)^m x(1-x)z^2 \frac{d}{dx} \Big(\sqrt{1-x} z^{2m+1}f_{2m}(x)\Big) \in \mathbb{Q}[\sqrt{1-x},x,z,z'].
\end{align*}
\end{cor}

For example, by direct calculations, we have
\begin{align}
&\sum_{n=0}^\infty \frac{(-1)^n(2n+1)^{3}}{\sinh^2\Big(\frac{2n+1}{2}y\Big)}=z^4x(1-x)\big(z-3z'(1-x)\big),\label{eq-cases-sinh-1}\\
&\sum_{n=0}^\infty \frac{(-1)^n(2n+1)^{3}}{\cosh^2\Big(\frac{2n+1}{2}y\Big)}= -\frac{1}{2} \sqrt{1-x}xz^4\big(6(x-1)z'+z\big).\label{eq-cases-cosh-1}
\end{align}
The identities \eqref{eq-cases-sinh-1} and \eqref{eq-cases-cosh-1} are the same as \cite[Eqs. (3.35) and (3.45)]{XuZhao-2022}.

We now turn to the similar sums as those in Theorem \ref{thm-nc-cases} but with even power of $2n+1$
replaced by the odd ones. These sums are closely related to the Bernoulli number.

\begin{thm} \label{thm-ds-cases}
For any positive integer $m\ge 1$,
\begin{align}\label{power-11}
\sum_{n=0}^\infty \frac{(2n+1)^{2m-1}}{e^{(2n+1)y}+1}=\frac{1}{8m}
\Big( (-1)^{m}q_{2m}(x)z^{2m} + 2(2^{2m-1}-1) B_{2m}\Big)\in \mathbb{Q}[x,z],
\end{align}
where $q_n(x)$ appears in the coefficients of the Maclaurin series of $u\, \ds(u)$ in \eqref{equ:cor-ds-fe-u}.
\end{thm}

\begin{proof}
The proof is completely similar to the proof of Theorem \ref{thm-dc-cases}. Applying \eqref{fe-ds} gives
\begin{align}\label{ds-fe-u}
u\ds(u)&=1-\sum_{m=1}^\infty \frac{(-1)^m2(2^{2m-1}-1)B_{2m}}{(2m)!z^{2m}}u^{2m}\nonumber\\
&\quad+4\sum_{m=1}^\infty \frac{(-1)^m}{(2m-1)!z^{2m}}\left\{\sum_{n=0}^\infty \frac{(2n+1)^{2m-1}}{e^{(2n+1)y}+1}\right\}u^{2m},
\end{align}
where we used the Laurent series expansions of $\csc(x)$ and power series expansions of $\sin(x)$
\begin{align*}
\csc(x)=\sum_{m=0}^\infty (-1)^{m-1}\frac{2(2^{2m-1}-1)B_{2m}}{(2m)!} x^{2m-1} \quad\text{and}\quad \sin(x)=-\sum_{m=1}^\infty \frac{(-1)^m}{(2m-1)!} x^{2m-1}.
\end{align*}
One the other hand, by Corollary \ref{cor-dc+ds} and comparing the coefficients of $u^{2m}$ of \eqref{ds-fe-u}
and \eqref{equ:cor-ds-fe-u} we arrive at the conclusion.
\end{proof}

\begin{cor}  \label{cor-coshWithq}
For any integer $m\ge 1$ we have
\begin{align*}
\sum_{n=0}^\infty \frac{(2n+1)^{2m}}{\cosh^2((2n+1)y/2)}
 &\,=\frac{(-1)^{m}}{2m} x(1-x)z^{2m+1}\Big(z q'_{2m}(x)+2mz'q_{2m}(x)\Big), \\
\sum_{n=0}^\infty \frac{(2n+1)^{2m+1}\sinh((2n+1)y/2)}{\cosh^3((2n+1)y/2))}
    &\,=\frac{(-1)^{m}}{2m}  x(1-x)z^2\frac{d}{dx}\bigg\{x(1-x)z^{2m+1}\Big(z q'_{2m}(x)+2mz'q_{2m}(x)\Big) \bigg\},
\end{align*}
where $q'_m(x):=\frac{dq_m(x)}{dx}$.
\end{cor}

\begin{proof} This follows quickly from Theorem \ref{thm-ds-cases} and the identities
\begin{align*}
\sum_{n=0}^\infty \frac{(2n+1)^{2m}}{\cosh^2((2n+1)y/2)}
    &\,=4x(1-x)z^2\frac{d}{dx} \sum_{n=0}^\infty \frac{(2n+1)^{2m-1}}{e^{(2n+1)y}+1},\\
\sum_{n=0}^\infty \frac{(2n+1)^{2m+1}\sinh((2n+1)y/2)}{\cosh^3((2n+1)y/2))}
    &\,=x(1-x)z^2\frac{d}{dx}\sum_{n=1}^\infty \frac{(2n+1)^{2m}}{\cosh^2((2n+1)y/2)}.
\end{align*}
\end{proof}

We can now prove a more precise version of Theorem \ref{thm-main2}.
\begin{thm} \label{thm-main2Precise}
Set $\Gamma=\Gamma(1/4)$. For any integer $m\ge 0$ we have
\begin{align*}
 \sum_{n=0}^\infty \frac{(2n+1)^{4m+2}}{\cosh^2((2n+1)\pi/2)}
    &\,=-\frac{q'_{4m+2}(1/2)}{2^{4m+7}(2m+1)} \cdot \frac{\Gamma^{8m+8}}{\pi^{6m+6}},\\
 \sum_{n=0}^\infty \frac{(2n+1)^{4m}}{\cosh^2((2n+1)\pi/2)}
    &\,= \frac{q_{4m}(1/2)}{2^{4m+1}} \cdot \frac{\Gamma^{8m}}{\pi^{6m+1}}, \\
\sum_{n=0}^\infty \frac{(2n+1)^{4m+1}\sinh((2n+1)\pi/2)}{\cosh^3((2n+1)\pi/2))}
    &\,=\frac{1}{4^{2 m+5} }
    \left\{
    \begin{array}{l}
    \Gamma^{8}\big( 4 q_{4m}(1/2) +q_{4m}''(1/2)/m\big) \\
     +   2^8 (4 m+1) q_{4m}(1/2) \pi^4 \phantom{\frac11}
    \end{array}
    \right\} \frac{ \Gamma^{8m}}{\pi^{6m+6}}, \\
\sum_{n=0}^\infty \frac{(2n+1)^{4m+3}\sinh((2n+1)\pi/2)}{\cosh^3((2n+1)\pi/2))}
    &\,=\frac{(4 m+3)q_{4m+2}'(1/2)}{2^{4m+7}(2 m+1)} \cdot \frac{\Gamma^{8m+8}}{\pi^{6m+7}} .
\end{align*}
Here, if $m=0$ then the term $q_{4m}''(1/2)/m$ does not appear in the third equation.
\end{thm}
\begin{proof} We have the following classical result (see the bottom of page 47 of Hancock's book \cite{Hancock1910}):
\begin{equation*}
\s(iu,\sqrt{1-k^2})=i\frac{\s(u,k)}{\cn(u,k)},\quad   \dn(iu,\sqrt{1-k^2})=\frac{\dn(u,k)}{\cn(u,k)},
\end{equation*}
where $i=\sqrt{-1}$ as usual. Therefore
\begin{equation} \label{dnCoeff}
iu\ds(iu,\sqrt{1-k^2})-u\ds(u,k)=0.
\end{equation}
This implies that for all $m>0$
\begin{equation}\label{equ:qVanish}
q_{4m-2}(1-x)+q_{4m-2}(x)=0.
\end{equation}
Differentiating \eqref{dnCoeff} with respect to $x$ we see that for all $m\ge 0$
\begin{equation}\label{equ:q'Vanish}
q_{4m}'(1-x)+q_{4m}'(x)=0.
\end{equation}
Differentiating \eqref{equ:qVanish} twice we get for all $m>0$
\begin{equation}\label{equ:q''Vanish}
q''_{4m-2}(1-x)+q''_{4m-2}(x)=0.
\end{equation}
Taking $x=1/2$ in \eqref{equ:qVanish}--\eqref{equ:q''Vanish} we have for all $m>0$
\begin{equation*}
 q_{4m-2}(1/2)=0, \quad q'_{4m}(1/2)=0, \quad q''_{4m-2}(1/2)=0.
\end{equation*}
Hence the theorem follows from Corollary \ref{cor-coshWithq} and \eqref{special-case-x}.
\end{proof}

Note that the above theorem confirms some identities contained in \cite[Conjecture 4.1]{XuZhao-2022}.
\begin{exa} We have
\begin{align*}
&\sum_{n=1}^\infty \frac{(2n-1)^{5}\sinh((2n-1)\pi/2)}{\cosh^3((2n-1)\pi/2))}=\frac{3\Gamma^{8}(1/4)}{64\pi^{8}}+\frac{\Gamma^{16}(1/4)}{3\cdot 2^{12}\pi^{12}},\\
&\sum_{n=1}^\infty \frac{(2n-1)^{7}\sinh((2n-1)\pi/2)}{\cosh^3((2n-1)\pi/2))}=\frac{3\Gamma^{16}(1/4)}{2^{10}\pi^{13}},\\
&\sum_{n=1}^\infty \frac{(2n-1)^{9}\sinh((2n-1)\pi/2)}{\cosh^3((2n-1)\pi/2))}=\frac{189\Gamma^{16}(1/4)}{5\cdot  2^{10}\pi^{14}}+\frac{9\Gamma^{24}(1/4)}{2^{16}\pi^{18}},\\
&\sum_{n=1}^\infty \frac{(2n-1)^{11}\sinh((2n-1)\pi/2)}{\cosh^3((2n-1)\pi/2))}=\frac{153\Gamma^{24}(1/4)}{2^{14}\pi^{19}}.
\end{align*}
\end{exa}

We can now prove a more precise evaluation of the second integral in Theorem \ref{thm-main34}.

\begin{thm}  \label{thm-main4Precise}
For any positive integer $p$,
\begin{align}
&\int_0^\infty \frac{x^{4p+1}}{(\cos x+\cosh x)^2}dx  \nonumber \\
&= \frac{(-1)^p}{2^{6p+11}p} \big( 4 p \, q_{4p}(1/2) +q_{4p}''(1/2)\big)\frac{\Gamma^{8p+8}(1/4)}{\pi^{2p+4}}
- \frac{(-1)^p (4 p+1) q_{4p}(1/2)}{2^{6p+3}} \frac{\Gamma^{8p}(1/4)}{\pi^{2p}}. \label{equ:conj-for-a1}
\end{align}
And if $p=0$, then
\begin{align*}
\int_0^\infty \frac{x}{(\cos x+\cosh x)^2}dx = -\frac{1}{2^3}+\frac{\Gamma^{8}(1/4)}{2^{9}\pi^4}.
\end{align*}
\end{thm}
\begin{proof}
By Cor. \ref{cor-Int-cos-h-1} we have
\begin{align*}
&\int_0^\infty \frac{x^{4p+1}dx}{(\cos x+\cosh x)^2}\nonumber\\
&=(-1)^p \frac{\pi^{4p+2}}{2^{2p+1}}\sum_{n=0}^\infty \frac{(2n+1)^{4p+1}\sinh\left(\frac{2n+1}{2}\pi\right)}{\cosh^3\left(\frac{2n+1}{2}\pi\right)}
 -(-1)^p(4p+1) \frac{\pi^{4p+1}}{2^{2p+1}}\sum_{n=0}^\infty \frac{(2n+1)^{4p}}{\cosh^2\left(\frac{2n+1}{2}\pi\right)}\\
&=(-1)^p \frac{\pi^{4p+2}}{2^{2p+1}} \frac{1}{4^{2 p+5} p}
    \left\{
    \begin{array}{l}
    \Gamma^{8}\big( 4 p \, q_{4p}(1/2) +q_{4p}''(1/2)\big) \\
     +   2^8 p  (4 p+1) q_{4p}(1/2) \pi^4 \phantom{\frac11}
    \end{array}
    \right\} \frac{ \Gamma^{8p}}{\pi^{6p+6}}
 -(-1)^p(4p+1) \frac{\pi^{4p+1}}{2^{2p+1}} \frac{q_{4p}(1/2)}{2^{4p+1}} \cdot \frac{\Gamma^{8p}}{\pi^{6p+1}}\\
&= \frac{(-1)^p}{2^{6p+11}}
    \Big\{
    \Gamma^{8}\big( 4 \, q_{4p}(1/2) +q_{4p}''(1/2)/p\big)
    - 2^8 (4 p+1) q_{4p}(1/2) \pi^4  \Big \} \frac{ \Gamma^{8p}}{\pi^{2p+4}}
\end{align*}
by Theorem \ref{thm-main2Precise} if $p>0$. If $p=0$ then the term $q_{4p}''(1/2)/p$ does not appear. Since $q_0(x)=1$
the case $p=0$ follows immediately.
\end{proof}

\begin{exa} \label{exa-Integral+}
By \emph{Mathematica} we can compute the following easily using the formulas in Theorem \ref{thm-main4Precise}:
\begin{align*}
&\int_0^\infty \frac{x^{5}}{(\cos x+\cosh x)^2}dx=\frac{3\Gamma^{8}(1/4)}{2^{9}\pi^{2}}-\frac{\Gamma^{16}(1/4)}{3\cdot 2^{15}\pi^{6}},\\
&\int_0^\infty \frac{x^{9}}{(\cos x+\cosh x)^2}dx=-\frac{189\Gamma^{16}(1/4)}{5\cdot 2^{15}\pi^{4}}+\frac{9\Gamma^{24}(1/4)}{2^{21}\pi^{8}},\\
&\int_0^\infty \frac{x^{13}}{(\cos x+\cosh x)^2}dx=\frac{18711\Gamma^{24}(1/4)}{5\cdot 2^{21}\pi^{6}}-\frac{5301\Gamma^{32}(1/4)}{7\cdot 2^{27}\pi^{10}},\\
&\int_0^\infty \frac{x^{17}}{(\cos x+\cosh x)^2}dx=-\frac{5544693 \Gamma^{32}(1/4)}{5\cdot 2^{27}\pi^{8}}+\frac{233361\Gamma^{40}(1/4)}{2^{33}\pi^{12}},\\
&\int_0^\infty \frac{x^{21}}{(\cos x+\cosh x)^2}dx=\frac{4233045663 \Gamma^{40}(1/4)}{5\cdot  2^{33} \pi ^{10}}-\frac{1940699817 \Gamma^{48}(1/4)}{11\cdot  2^{39} \pi ^{14}},\\
&\int_0^\infty \frac{x^{25}}{(\cos x+\cosh x)^2}dx=-\frac{17651016517365 \Gamma^{48}(1/4)}{13\cdot  2^{39}\pi ^{12}}+\frac{283641186969 \Gamma^{56}(1/4)}{2^{45} \pi ^{16}},\\
&\int_0^\infty \frac{x^{29}}{(\cos x+\cosh x)^2}dx=\frac{20473552886711463 \Gamma^{56}(1/4)}{5\cdot 2^{45}\pi ^{14}}-\frac{854871935822163 \Gamma^{64}(1/4)}{2^{51}\pi^{18}}.
\end{align*}
\end{exa}

We conclude the paper with the following conjecture which is supported by comparing the coefficients in Examples \ref{exa-Integral-} and \ref{exa-Integral+}.
\begin{con}\label{conj:finalConj}
For all positive integer $p$ let $i_p,j_p,g_p$ and $h_p$ be the numbers appearing in Theorems \ref{thm-main3Precise} and \ref{thm-main4Precise}. Then the following relations hold:
\begin{align*}
 g_p=-\frac{2^{2p-1}-(-1)^p}{2^{2p-1}}i_p\quad \text{and}\quad h_p=-\frac{2^{2p-1}+(-1)^p}{2^{2p-1}}j_p.
\end{align*}
\end{con}
Note that our approach to the two theorems is very different so that it is highly likely that there is an intimate relation between the two we have not discovered yet.

\medskip

\noindent
\emph{Note added after proof.} After the paper was accepted, Srinivasa Raghava informed us that he could confirm Conjecture \ref{conj:finalConj} by applying some Hecke correspondence which shows the modular source of this conjecture. 

\medskip
{\bf Acknowledgments.} The authors sincerely thank the anonymous reviewer for the valuable comments, which helped improve the quality of this paper. C. Xu is supported by the General Program of Natural Science Foundation of Anhui Province (Grant No. 2508085MA014). J. Zhao was supported by the Jacobs Prize from the Bishop's school. Both authors thank the Tianyuan Mathematical Center in Southeast China, funding number 12526102.

\end{document}